\documentclass[pdflatex,sn-mathphys-ay, iicol]{sn-jnl}% Math and Physical Sciences Numbered Reference Style
\usepackage{graphicx}%
\usepackage{multirow}%
\usepackage{amsmath,amssymb,amsfonts}%
\usepackage{amsthm}%
\usepackage{mathrsfs}%
\usepackage[title]{appendix}%
\usepackage{xcolor}%
\usepackage{textcomp}%
\usepackage{manyfoot}%
\usepackage{booktabs}%
\usepackage{algorithm}%
\usepackage{algorithmicx}%
\usepackage{algpseudocode}%
\usepackage{listings}%
\usepackage{subcaption, mathtools, enumitem}
\usepackage[most]{tcolorbox}

\DeclareMathOperator*{\argmax}{arg\,max}
\definecolor{kthblue}{RGB}{25,84,166}
\definecolor{myred}{RGB}{194, 50, 72}
\definecolor{mygreen}{RGB}{0, 190, 151}
\theoremstyle{thmstyleone}%
\newtheorem{theorem}{Theorem}%  meant for continuous numbers
\newtheorem{proposition}[theorem]{Proposition}% 
\newtheorem{corollary}{Corollary}

\theoremstyle{thmstyletwo}%
\newtheorem{remark}{Remark}%

\theoremstyle{thmstylethree}%

\newtcbtheorem{tuning}{Tuning scheme}{coltitle=black,colback=white,colframe=gray!30,fonttitle=\bfseries, width=\columnwidth, boxrule=0.3mm, sharp corners}{tune}

\begin{document}

\title[Large deviation-based tuning schemes for MH algorithms]{\center{Large deviation-based tuning schemes \\for Metropolis-Hastings algorithms}}

\author*[1]{\fnm{Federica} \sur{Milinanni}}\email{fedmil@kth.se}



\affil*[1]{\orgdiv{Department of Mathematics}, \orgname{KTH Royal Institute of Technology}, \orgaddress{\street{Lindstedtsvägen 25}, \city{Stockholm}, \postcode{11428}, %\state{State}, 
\country{Sweden}}}

%\affil[2]{\orgdiv{Department}, \orgname{Organization}, \orgaddress{\street{Street}, \city{City}, \postcode{10587}, \state{State}, \country{Country}}}

%\affil[3]{\orgdiv{Department}, \orgname{Organization}, \orgaddress{\street{Street}, \city{City}, \postcode{610101}, \state{State}, \country{Country}}}


\abstract{Markov chain Monte Carlo (MCMC) methods are one of the most popular classes of algorithms for sampling from a target probability distribution.
A rising trend in recent years consists in analyzing the convergence of MCMC algorithms using tools from the theory of large deviations. 
In (Milinanni \& Nyquist, 2024), a new framework based on this approach has been developed to study the convergence of empirical measures associated with algorithms of Metropolis-Hastings type, a broad and popular sub-class of MCMC methods.

The goal of this paper is to leverage these large deviation results to improve the efficiency of Metropolis-Hastings algorithms. Specifically, we use the large deviations rate function (a central object in large deviation theory) to quantify and characterize the algorithms' speed of convergence. We begin by extending the analysis from (Milinanni \& Nyquist, 2024), deriving alternative representations of the rate function. Building on this, we establish explicit upper and lower bounds, which we then use to design schemes to tune Metropolis-Hastings algorithms.}




\maketitle
\section{Introduction}
Sampling from a target probability distribution $\pi$ is a fundamental task in many applied fields, including, e.g., computational biology, statistical mechanics, epidemiology and ecology \cite{andrieu2003introduction, asmussen2007stochastic, robert2004monte}. A setting where sampling is predominant is that of Bayesian statistics. In fact, a direct computation of the Bayesian posterior distribution in practical applications is typically computationally prohibitive, and a valid alternative to direct computation of the posterior consists in sampling from such probability distribution using sampling algorithms.

Markov chain Monte Carlo (MCMC) methods are a wide and popular class of sampling algorithms. These methods generate a Markov chain that has the target $\pi$ as invariant probability distribution \cite{robert2004monte}. One of the most common algorithms in this class is the Metropolis-Hastings (MH) algorithm, which is based on a mechanism of proposing a state for the next element of the chain, and then accepting or rejecting such proposal. Different choices of proposal mechanism lead to different algorithms of MH type \cite{metropolis1953equation, hastings1970monte}. Examples of algorithms that fall within the MH category are the Independent Metropolis-Hastings (IMH) algorithm, the Random Walk Metropolis (RWM) algorithm, the Metropolis-adjusted Langevin algorithm (MALA), and Hamiltonian Monte Carlo (HMC) \cite{Tie94, robert2004monte, mengersen1996rates, sherlock2010random, besag1994comments, roberts1998optimal, roberts1996exponential, duane1987hybrid}.

In order to design efficient sampling algorithms, it is of crucial importance to perform a theoretical analysis of the underlying Markov processes. In fact, in applied problems a blind use of off-the-shelf sampling algorithms can require excessively large computation times. A theoretical understanding of the convergence properties of MCMC methods helps us in designing algorithms that successfully perform the sampling task with a reasonable computational effort. The Monte Carlo literature already provides a variety of tools to perform convergence analysis of MCMC algorithms. Some of the more classical tools include: spectral gap, asymptotic variance, mixing times, and functional inequalities \cite{bedard2008optimal, rosenthal2003asymptotic, DHN00, franke2010behavior, frigessi1993convergence, hwang2005accelerating, ALP+22a, ALP+23a, PSW24}. 

In recent years, we have witnessed a growing trend toward using the theory of large deviations to analyze the speed of convergence of sampling algorithms. This approach is based on the seminal work by Donsker and Varadhan \cite{donsker1975asymptotic, DV75b, donsker1975asymptotic3}, and consists in deriving a \textit{large deviation principle} (LDP) for the \textit{empirical measures} associated with the algorithms' underlying Markov processes. More in detail, if an algorithm generates discrete time Markov chains $\{X_i\}_{i\in\mathbb{N}}$, as the algorithms of MH type, the corresponding empirical measure is the random probability measure defined as 
\begin{align*}
    L^n(dx)=\frac{1}{n}\sum_{i=0}^{n-1}\delta_{X_i}(dx).
\end{align*}
If the algorithm is well-designed, the sequence of empirical measures $\{L^n\}$ converges almost surely to the target $\pi$, in the weak topology. 
By saying that the sequence of empirical measures $\{L^n\}$ satisfies a large deviation principle with speed $n$ and some rate function $I$, we mean, roughly speaking (see Section~\ref{sec:ldp} for a formal definition), that the convergence $\lim_{n\to\infty}L^n=\pi$ is exponentially fast. In addition, the rate function $I$ provides important information on the corresponding speed of convergence. The fundamental concept is: \textit{a larger rate function indicates faster convergence of the sampling algorithm}.

Whereas more classical convergence tools usually describe the distribution of the $n$-th iterate of the Markov chain, $X_n$, this large deviations approach accounts for the information on the time averaging effect of the empirical measure. This distinction is crucial because, in practice, Monte Carlo estimates rely on the empirical measure $L_n$ rather than only on $X_n$. For example, we often approximate the integral $\int f(x)\pi(dx)$ with the empirical average $\frac{1}{n}\sum_{i=0}^{n-1}f(X_i)$. Notably, this empirical average corresponds to integrating the function $f$ with respect to the empirical measure, i.e., our approximation is $\int f(x)\pi(dx)\approx \int  f(x)L^n(dx)$.

The first works on large deviations for sampling algorithms appeared in the early 2010s \cite{PDD+11, dupuis2012infinite}. In these papers, the authors study \textit{parallel tempering}, an algorithm typically used to sample from multimodal distributions. Remarkably, the use of the large deviations approach in the analysis of the parallel tempering algorithm led to the design of a new sampling scheme, the \textit{infinite swapping algorithm}. Later works in this line of research include \cite{rey2015irreversible, rey2015variance, rey2016improving, bierkens2016non}, where the empirical measure large deviation principle is used to show that algorithms based on non-reversible Markov processes exhibit faster convergence. Further studies on the parallel tempering algorithm, and the related infinite swapping algorithm, are carried out in \cite{doll2018large} and \cite{DW22}. In the former, the large deviation rate function is used to derive convergence properties of the two methods, while in the latter, the temperatures (hyperparameters) of these algorithms are tuned via the empirical measure LDP. In \cite{bierkens2021large} the empirical measure large deviation is used to study and optimize the zig-zag sampler.

The first study of algorithms of MH type in continuous state space using large deviations techniques is presented in \cite{milinanni2024a}. Therein, it is proved a large deviation principle for general discrete-time stochastic processes, of which the MH algorithms are a special case. In a later work \cite{milinanni2024b}, the LDP from \cite{milinanni2024a} is applied to specific algorithms from the MH class, specifically, the IMH, MALA and the RWM algorithms.

The goal of this paper is to leverage the LDP derived in \cite{milinanni2024a} to tune algorithms of MH type to achieve faster convergence. Since important information on the algorithms' speed of convergence can be derived from the large deviation rate function $I$, the first part of this paper is dedicated to analyzing this object, complementing the results from \cite{milinanni2024a}.

The large deviation rate function in \cite{milinanni2024a} is characterized as the minimum of a relative entropy functional (also referred to as Kullback-Leibler divergence), where the minimum is taken over a specific type of probability measures. Since optimization over probability measures can be highly demanding, in this paper we provide an alternative representation of the rate function, which is dual to that in \cite{milinanni2024a}. This new representation expresses the rate function as the solution to a maximization problem over \textit{continuous functions} that satisfy certain properties. A key advantage of this alternative representation is that optimizing over continuous functions is often more tractable compared to optimizing over probability measures. In addition, we also provide an equivalent version of this representation, where the maximization is over \textit{measurable functions} satisfying certain properties. All representations are presented together in Theorem~\ref{thm:update}.

Based on the various representations of the large deviation rate function $I$, we derive explicit upper and lower bounds for $I$ that are easier to compute numerically than the rate function itself. These bounds allow us to construct intervals that approximate $I$, enabling us to quantify the rate of convergence of MH algorithms, avoiding the costly exact computation of the rate function.

Building on these bounds, we design three large deviation-based tuning schemes aimed at identifying ``near-optimal'' hyperparameters (in the sense of leading to fastest convergence) for sampling algorithms of MH type. We also present an illustrative example demonstrating the application of one of these schemes to tune the IMH algorithm. 
To our knowledge, these are the first results on tuning algorithms of Metropolis-Hastings type based on the theory of large deviations. In future work, we will consider more advanced algorithms of MH type, such as MALA and HMC. The ultimate goal is to develop large deviation-based tuning schemes that can be used to pre-calibrate Metropolis-Hastings algorithms in applications, where limited information on the target $\pi$ is typically available.

The remainder of the paper is organized as follows. In Sections~\ref{sec:notation}, \ref{subsec:MH}, and~\ref{sec:ldp}, we describe the notation used throughout the paper, the algorithms of Metropolis-Hastings type, and the large deviations result from \cite{milinanni2024a}, respectively. In Section~\ref{sec:alternative}, we prove two alternative representations of the LDP rate function and gather all representations in Theorem~\ref{thm:update}. In Section~\ref{sec:bounds}, we derive upper and lower bounds for the rate function. We then use these results to design three large deviation-based tuning schemes, described in Section~\ref{sec:tuning}. The paper ends with an illustrative example in Section~\ref{sec:IMH} on the use of one of the tuning schemes to optimize the hyperparameters in the IMH algorithm. 

\section{Preliminaries}
\subsection{Notation}
\label{sec:notation}
Throughout this paper we will work on some probability space $(\Omega, \mathcal{F},\mathbb{P})$.

Given a Polish space $S$, $\mathcal{C}(S)$ denotes the space of continuous functions on $S$, and $\mathcal{M}(S)$ the space of measurable functions on $S$. As Polish space $S$ we will consider continuous subsets of $\mathbb{R}^d$, $d\ge1$.

The space of probability measures on a Polish space $S$ will be indicated as $\mathcal{P}(S)$, and it will be metrized through the Lévy-Prohorov distance $\text{dist}_{LP}$. For any $\varepsilon>0$, $\mu\in\mathcal{P}(S)$,
\begin{align*}
    B_\varepsilon(\mu)=\{\nu\in\mathcal{P}(S)\,|\,\text{dist}_{LP}(\mu,\nu)<\varepsilon\}
\end{align*}
denotes the ball of radius $\varepsilon$ centered at $\mu$ in the Lévy-Prohorov metric.

With a slight abuse of notation, given a measure $\mu(dx)\in\mathcal{P}(S)$, we will denote by $\mu(x)$ its density with respect to the Lebesgue measure (if it exists). 

Given $\theta,\varphi\in\mathcal{P}(S)$, we denote the \textit{relative entropy} between $\theta$ and $\varphi$ as
\begin{equation*}
    R(\theta\parallel\varphi)=\begin{cases}
        \int_{S}\log\frac{d\theta}{d\varphi}d\theta,\quad\text{if }\theta\ll\varphi,\\
        +\infty,\qquad\qquad \text{ otherwise}.
    \end{cases}
\end{equation*}

For $\gamma\in\mathcal{P}(S\times S)$, let $[\gamma]_1(dx)$ and $[\gamma]_2(dx)$ denote the first and second marginals of $\gamma(dx,dy)$, respectively, and define 
\begin{equation}
    \label{def:Amu}
    A(\mu)=\left\{\gamma\in\mathcal{P}(S\times S)\, | \,[\gamma]_1=[\gamma]_2=\mu\right\}.
\end{equation}

For a measurable space, let $q(x,\cdot)\in\mathcal{P}(S)$ be a collection of probability measures parametrized by $x\in S$. We call $q$ a stochastic kernel on $S$ if, for every measurable set $A\in\mathcal{B}(S)$, $x\mapsto q(x,A)\in[0,1]$ is a measurable function. 

Given a probability measure $\mu\in\mathcal{P}(S)$ and a stochastic kernel $q(x,dy)$ on $S$, we say that $\mu$ is invariant for $q(x,dy)$ if, for every measurable set $A\in\mathcal{B}(S)$,
 \begin{align*}
     \mu(A)=\int_Sq(x,A)\mu(dx).
 \end{align*}
 We define $\mathcal{Q}(\mu)$ as the set of all stochastic kernels $q(x,dy)$ on $S$ such that $\mu$ is an invariant distribution for $q$.

 Given any stochastic kernel $q(x,dy)$ and a function $u:S\to \mathbb{R}$, we define
\begin{equation*}
    (qu)(x)=\int_Su(y)q(x,dy).
\end{equation*}

\subsection{The Metropolis-Hastings algorithm}
\label{subsec:MH}
Given a target probability distribution $\pi\in\mathcal{P}(S)$, the Metropolis-Hastings (MH) algorithm produces a discrete time Markov chain $\{X_i\}_{i\in\mathbb{N}}$ that has $\pi$ as invariant distribution. An essential element of the MH algorithm is the proposal distribution $J(\cdot|x)\in\mathcal{P}(S)$, defined for all $x\in S$. Different choices for the proposal $J$ lead to different MH algorithms. 

The MH algorithm is an iterative algorithm, illustrated in Algorithm~\ref{alg:MH}. At iteration $i$, assume that the state of the chain is in $X_i=x_i\in S$. The MH algorithm generates a proposal $Y_{i+1}$ for the next state of the chain, $X_{i+1}$, by sampling from $J(\cdot|x_i)\in\mathcal{P}(S)$. The proposal $Y_{i+1}$ is then accepted or rejected based on the \textit{Hastings ratio}
\begin{align*}
    \varpi(x,y)=\min\left\{1,\frac{\pi(y)J(x|y)}{\pi(x)J(y|x)}\right\}.
\end{align*}
Specifically, with probability $\varpi(x_i,Y_{i+1})$ we accept the proposal and set $X_{i+1}=Y_{i+1}$. Under rejection, i.e., with probability $1-\varpi(x_i,Y_{i+1})$, we set $X_{i+1}=x_i$.

\begin{algorithm}
\caption{Metropolis-Hastings algorithm ($i$-th iteration)}
\label{alg:MH}
\begin{algorithmic}[1]
\Require Given $X_i=x_i$,
 \State Generate a proposal $Y_{i+1}\sim J(\cdot|x_i)$
 \State Set
    \begin{equation*}
        X_{i+1}=\begin{cases}
            Y_{i+1}\quad \text{with probability }\varpi(x_i,Y_{i+1})\\
            x_i\quad \text{with probability }1-\varpi(x_i,Y_{i+1})
        \end{cases}
    \end{equation*}
\end{algorithmic}
\end{algorithm}

Algorithm~\ref{alg:MH} generates a Markov chain characterized by the following stochastic transition kernel:
\begin{equation}
\label{MHtransitionKernel}
    K(x,dy)=a(x,y)dy+r(x)\delta_x(dy),
\end{equation}
where 
\begin{equation}
\label{MHacceptance}
    a(x,y)=\min\left\{1,\frac{\pi(y)J(x|y)}{\pi(x)J(y|x)}\right\}J(y|x)
\end{equation}
and
\begin{equation}
\label{MHrejection}
    r(x)=1-\int_Sa(x,y)dy.
\end{equation}

\subsection{Large deviation principle for Metropolis-Hastings Markov chains}
\label{sec:ldp}
Large deviation theory studies the asymptotic rate of decay of rare events, which is quantified via a \textit{large deviation principle} (LDP).

Let $\mathcal{Y}$ be a Polish space and $\{Y_n\}\subset \mathcal{Y}$ a sequence of random variables on $\mathcal{Y}$. We say that $\{Y_n\}$ satisfies an LDP on $\mathcal{Y}$ with \textit{speed} $n$ and \textit{rate function} $I:\mathcal{Y}\to[0,\infty]$ if $I$ has compact sub-level sets and, for every measurable set $A\subset\mathcal{Y}$, the following inequalities are satisfied:
\begin{equation}
\label{inequality_in_general}
\begin{split}
        -\inf_{y\in A^\circ}I(y)&\le \liminf_{n\to\infty}\frac{1}{n}\log\mathbb{P}(Y^n\in A^\circ)\\
    &\le \limsup_{n\to\infty}\frac{1}{n}\log\mathbb{P}(Y^n\in \bar A)\\
    &\le -\inf_{y\in \bar A}I(y),
\end{split}
\end{equation}
where $A^\circ$ and $\bar A$ denote the interior and the closure of $A$, respectively.

Roughly speaking, these inequalities suggest that 
\begin{align*}
    \mathbb{P}(Y^n\in A)\simeq e^{-n\cdot \inf_{y\in A}I(y)}.
\end{align*}
If $\inf_{y\in A}I(y)>0$, this indicates that $\lim_{n\to\infty}\mathbb{P}(Y^n\in A)=0$, that is, the events $(Y^n\in A)$ are more and more rare as $n\to+\infty$. By means of the rate function $I$, the LDP allows us to quantify the exponential rate of decay of the probability of the rare event $(Y^n\in A)$.

In \cite{milinanni2024a}, a new framework based on large deviation theory was developed to describe the convergence of stochastic processes arising from algorithms of MH type. Here, we summarize this approach and the large deviation result upon which this paper builds.

Let $\{X_i\}_{i\ge0}$ be a discrete time Markov chain. For every $n\in\mathbb{N}$, the \textit{empirical measure} of the first $n$ elements of the chain is the random probability distribution $L^n\in\mathcal{P}(S)$ defined as 
\begin{equation*}
    L^n(dx)=\frac{1}{n}\sum_{i=0}^{n-1}\delta_{X_i}(dx).
\end{equation*}
Consider now Markov chains $\{X_i\}$ associated with an algorithm of MH type. Under mild assumptions on the target distribution $\pi$ and the proposal $J(\cdot|x)$, the corresponding sequence of empirical measures $\{L^n\}$ converges almost surely to $\pi$ in the weak topology, i.e.,
\begin{align}
\label{weak_convergence}
    \lim_{n\to\infty}L^n=\pi,\qquad\text{a.s.}
\end{align}

This convergence implies that for any set $A\subset \mathcal{P}(S)$ bounded away from $\pi$ (that is, $\text{dist}_{LP}(A,\pi)>\varepsilon$ for some $\varepsilon>0$),
\begin{align}
\label{lim_prob_A}
    \lim_{n\to\infty}\mathbb{P}(L^n\in A)=0.
\end{align}

The main result in \cite{milinanni2024a}, stated below, characterizes the rate of convergence of \eqref{lim_prob_A} by means of a large deviation principle.

\begin{theorem}[Theorem 4.1 in \cite{milinanni2024a}]
    \label{TheoremMN}
    Let $\{X_i\}_{i\ge0}$ be the Metropolis–Hastings chain and $K(x,dy)$ the associated transition kernel. Let $\{L^n\}_{n\ge1}\subset\mathcal{P}(S)$ be the corresponding sequence of empirical measures. Under Assumptions (A.1)–(A.3) (see \cite[Section 3]{milinanni2024a}), $\{L^n\}_{n\ge1}$ satisfies a large deviation principle with speed $n$ and rate function
\begin{align}
\label{oldRateFunction}
    I(\mu)&=\inf_{\gamma\in A(\mu)}R(\gamma\parallel \mu\otimes K).
\end{align}
\end{theorem}

\begin{remark}
\label{rmk:rateFunc_kernels}
    By the chain rule for the relative entropy (see \cite[Theorem 2.6]{budhiraja2019analysis}), the rate function in \eqref{oldRateFunction} can be rewritten as
\begin{align}
\label{rateFuncWithQ}
    I(\mu) = \inf_{ q\in\mathcal{Q}(\mu)}\int_SR( q(x,\cdot)\parallel K(x,\cdot))\mu(dx),
\end{align}
where $\mathcal{Q}(\mu)$ denotes the set of probability kernels $q(x,dy)$ that have $\mu$ as invariant measure.
\end{remark}

Among the assumptions of Theorem~\ref{TheoremMN}, we suppose that the state space $S$ is a continuous subset of $\mathbb{R}^d$, and that both the target $\pi$ and the MH proposal distributions $J(\cdot|x)$ are absolutely continuous with respect to the Lebesgue measure on $S$.

Recalling \eqref{inequality_in_general}, the LDP in Theorem~\ref{TheoremMN} implies that for any measurable set $A\subset \mathcal{P}(S)$,
\begin{equation*}
\begin{split}
    -\inf_{\mu\in A^\circ}I(\mu)&\le \liminf_{n\to\infty}\frac{1}{n}\log\mathbb{P}(L^n\in A^\circ)\\
    &\le \limsup_{n\to\infty}\frac{1}{n}\log\mathbb{P}(L^n\in \bar A)\\
    &\le -\inf_{\mu\in \bar A}I(\mu),
\end{split}
\end{equation*}
giving the following asymptotic decay of probability: 
\begin{align}
\label{asymptotics}
    \mathbb{P}(L^n\in A)\simeq \exp\left\{-n\,\cdot\inf_{\mu\in A}I(\mu)\right\}.
\end{align}

If we choose $A = B_\varepsilon(\pi)^\complement$, for some $\varepsilon>0$, then \eqref{asymptotics} becomes
\begin{equation*}
%\label{asymptotics_B}
\begin{split}
\mathbb{P}\big(L^n \in B_\varepsilon(\pi&)^\complement\big)= \mathbb{P}(\text{dist}_{LP}(L^n,\pi)\ge\varepsilon)\\
&\simeq \exp\left\{-n\,\cdot\inf_{\mu\in B_\varepsilon(\pi)^\complement} I(\mu)\right\},
\end{split}   
\end{equation*}
therefore, the probability that the empirical measure $L^n$ is close to the target $\pi$ by $\varepsilon$ ($\text{dist}_{LP}(L^n,\pi)<\varepsilon$) satisfies
\begin{align}
\label{asymptotics_B}
    \mathbb{P}\big(L^n \in B_\varepsilon(\pi&)\big)\simeq 1-\exp\left\{-n\cdot\inf_{\mu\in B_\varepsilon(\pi)^\complement} I(\mu)\right\}
\end{align}
For an intuitive illustration, see Figure~\ref{fig:ball}.

\begin{figure*}[t]
\centering
\includegraphics[width=0.8\textwidth]{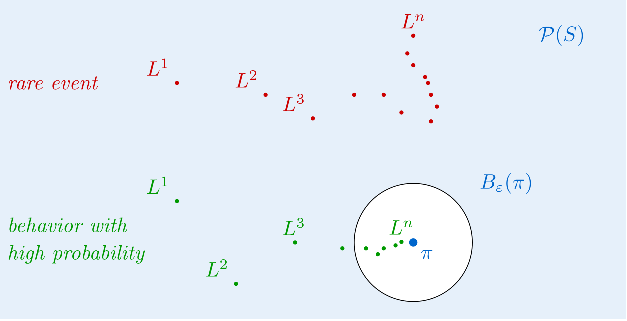}
    \caption{Illustration of the space $\mathcal{P}(S)$ of probability measures on a Polish space $S$, the target measure $\pi\in\mathcal{P}(S)$, the complement of the ball of radius $\varepsilon>0$ centered at $\pi$, $B_\varepsilon(\pi)^\complement$, and two random sequences of empirical measures $\{L^n\}$ corresponding to the behaviors with high and low probability (green and red dots, respectively). With high probability, for large $n$, the random empirical measure $L^n$ will be close to $\pi$, thus, $\mathbb{P}(L^n\in B_\varepsilon(\pi))\approx 1$ (green dots). Instead, $L^n\in B_\varepsilon(\pi)^\complement$ is a rare event (red dots) as $n$ grows, and the corresponding probability $\mathbb{P}(L^n\in B_\varepsilon(\pi)^\complement)$ decays to $0$ exponentially fast, as $n\to\infty$.}
    \label{fig:ball}
\end{figure*}

If follows from \eqref{asymptotics_B} that the larger the quantity
\begin{align}
\label{min_rateFunc}
    \inf_{\mu\in B_\varepsilon(\pi)^\complement} I(\mu)
\end{align}
is, the faster  $\mathbb{P}(\text{dist}_{LP}(L^n,\pi)<\varepsilon)$ converges to $1$. Thus, when designing an algorithm of MH type, we aim to maximize the quantity \eqref{min_rateFunc} to guarantee fast convergence to the target $\pi$. 

We end this section with the following proposition, showing that the optimization problem in \eqref{min_rateFunc} can be restricted to the boundary of the ball.

\begin{proposition}
\label{prop:min_on_sphere}
    Let
       \begin{align*}
       \partial B_\varepsilon(\pi) = \{\mu\in\mathcal{P}(S)\,|\, \textnormal{dist}_{LP}(\mu,\pi)=\varepsilon\}
   \end{align*}
   be the boundary of the ball $B_\varepsilon(\pi)$. The infimum in \eqref{min_rateFunc} satisfies
\begin{align}
\label{eq:equality_inf}
    \inf_{\mu\in B_\varepsilon(\pi)^\complement} I(\mu) = \inf_{\mu\in \partial B_\varepsilon(\pi)} I(\mu).
\end{align} 
\end{proposition}
\begin{proof}

    We now prove that for any $\mu\in B_\varepsilon(\pi)^\complement$ there exists a $\bar\mu\in\partial B_\varepsilon(\pi)$ such that $I(\bar\mu)\le I(\mu)$. Once we establish this, \eqref{eq:equality_inf}  follows directly.

    For this proof, we recall the properties of the rate function discussed in \cite{milinanni2024a}: the rate function is convex, $I(\mu)\ge0$ for all $\mu\in\mathcal{P}(S)$, and $I(\mu)=0$ if and only if $\mu=\pi$.

    Given $\mu\in B_\varepsilon(\pi)^\complement$, define
    \begin{align*}
        \mu_t = t\,\mu+(1-t)\,\pi
    \end{align*}
    and $f:[0,1]\to[0,+\infty)$ as
    \begin{align*}
        f(t) = \text{dist}_{LP}(\mu_t,\pi).
    \end{align*}
    Note that $f$ is continuous,
    \begin{align*}
        \lim_{t\to0^+}f(t)=  \text{dist}_{LP}(\pi,\pi)=0,
    \end{align*}
    and, because $\mu\in B_\varepsilon(\pi)^\complement$,
    \begin{align*}
        \lim_{t\to1^-}f(t)=  \text{dist}_{LP}(\mu,\pi)\ge \varepsilon.
    \end{align*}
    By continuity of $f$, there exists $\bar t\in [0,1]$ such that 
    \begin{align*}
        \text{dist}_{LP}(\mu_{\bar t},\pi)=\varepsilon,
    \end{align*}
    that is, $\bar\mu=\mu_{\bar t}\in \partial B_\varepsilon(\pi)$.

    Because the rate function is convex, $t\in[0,1]$ and $I(\pi)=0$,
    \begin{align*}
        I(\mu_{\bar t})&=I(\bar t\,\mu+(1-\bar t)\,\pi)\le \bar t\,I(\mu)+(1-\bar t)I(\pi)\\
        &=\bar tI(\mu)\le I(\mu),
    \end{align*}
    completing the proof.
\end{proof}

\section{Alternative representation of the rate function}
\label{sec:alternative}
In this section, we show that the rate function of the LDP from Theorem~\ref{TheoremMN} admits an alternative representation. This representation is dual to \eqref{oldRateFunction}. In \eqref{oldRateFunction}, $I(\mu)$ is obtained by solving a minimization problem over probability measures $\gamma\in A(\mu)$ (or, equivalently, over stochastic kernels $q\in\mathcal{Q}(\mu)$). Instead, in the following Proposition, we characterize $I(\mu)$ as a minimization problem over a certain class of continuous functions.

\begin{proposition}
\label{thm:alternativeRepresentation}
Let $\mathcal{U}$ be the set of continuous functions on $S$ that are bounded away from $0$ and $+\infty$, i.e.,
\begin{align*}
    \mathcal{U}=\big\{u\in\mathcal{C}(S)\,|\,&\exists \, c_l,c_u\in(0,+\infty) \\
    & \text{ s.t. }\;c_l \le u(x)\le c_u,\;\forall x\in S\big\},
\end{align*}
and define
\begin{align*}
    I^*(\mu)=-\inf_{u\in\mathcal{U}}\int_S\log\frac{Ku}{u}d\mu.
\end{align*}
The rate function $I(\cdot)$ in Theorem~\ref{TheoremMN} satisfies
\begin{align*}
    I(\mu)=I^*(\mu),
\end{align*}
for all $\mu\in\mathcal{P}(S)$.
\end{proposition}

\begin{proof}
Let
\begin{align*}
    \mathcal{M}_{\alpha,\beta}=\{\gamma\in\mathcal{P}(S\times S)\;\text{ s.t.} \;[\gamma]_1=\alpha, [\gamma]_2=\beta\},
\end{align*}
and note that $A(\mu)=\mathcal{M}_{\mu,\mu}$. Let 
    \begin{align*}
    \mathcal{V}=\big\{u\in\mathcal{C}&(S\times S)\,|\,\exists \, c_l, c_u \in(0,+\infty) \;\\
    &\text{ s.t. }\; c_l \le u(x,y)\le c_u,\;\forall (x,y)\in S\times S\big\}.
\end{align*}

    In this proof, we will apply Theorem 2.1 in \cite{donsker1975asymptotic3}, which states that for all $\alpha,\beta \in \mathcal{P}(S)$,
    \begin{align}
    \label{theorem2.1DV}
        I(\alpha,\beta)=J(\alpha,\beta),
    \end{align}
    where
    \begin{align*}
        &I(\alpha,\beta)\\
        &=-\inf_{u\in \mathcal{U}}\left[\int \log(Ku)(x)\alpha(dx)-\int\log u(x)\beta(dx)\right],
    \end{align*}
    and
    \begin{align*}
        &J(\alpha,\beta)  \\
        &= \inf_{\lambda\in\mathcal{M}_{\alpha,\beta}}\Bigg\{-\inf_{u\in\mathcal{V}}\Bigg[\log\iint u(x,y)K(x,dy)\alpha(dx)\\
        &  \hspace*{3cm} -\iint \log u(x,y)\lambda(\alpha,\beta)\Bigg]\Bigg\}.
    \end{align*}

    Let $\alpha=\beta=\mu$. Note that $I(\mu,\mu) = I^*(\mu)$. Since $\mathcal{M}_{\mu,\mu}=A(\mu)$, then, 
    \begin{align*}
        &J(\mu,\mu) \\
        &= \inf_{\gamma \in A(\mu)}\Bigg\{-\inf_{u\in\mathcal{V}}\Bigg[\log\iint u(x,y)K(x,dy)\mu(dx)\\
        &\hspace*{3cm} -\iint \log u(x,y)\gamma(dx,dy)\Bigg]\Bigg\}.
    \end{align*}
    
    By the Donsker-Varadhan variational formula for the relative entropy (see \cite[Lemma 2.4]{budhiraja2019analysis}), for all $\gamma,\theta\in\mathcal{P}(S\times S)$,
    \begin{align*}
        R(\gamma\parallel \theta)&=\sup_{g\in \mathcal{C}_b(S\times S)}\Bigg[\iint g(x,y)\gamma(dx,dy)\\
        & \hspace*{2cm} -\log\iint e^{g(x,y)}\theta(dx,dy)\Bigg]\\
        &=-\inf_{u\in\mathcal{V}}\Bigg[\log\iint u(x,y)\theta(dx,dy)\\
        & \hspace*{2cm} -\iint\log u(x,y)\gamma(dx,dy)\Bigg],
    \end{align*}
    where we let $u=e^{g}$. Recalling \eqref{oldRateFunction}, it follows that 
    \begin{align*}
        I(\mu)=\inf_{\gamma\in A(\mu)}R(\gamma\parallel \mu\otimes K)= J(\mu,\mu).
    \end{align*}
    By Donsker and Varadhan's result \eqref{theorem2.1DV}, $J(\mu,\mu)=I(\mu,\mu)$, which in turn is equal to $I^*(\mu)$. Thus, we obtain that $I(\mu)=I^*(\mu)$, which completes the proof.
\end{proof}

The previous result allows us to represent the large deviation rate function as an optimization problem over continuous functions bounded away from zero and infinity. In the following Proposition, we prove that such problem is equivalent to optimizing over \textit{measurable functions} bounded away from zero and infinity.

\begin{proposition}
    \label{prop:measurable_func}
    Let $\mathcal{U}'$ be the set of measurable functions on $S$ that are bounded away from $0$ and $+\infty$, i.e.,
\begin{align*}
    \mathcal{U}'=\big\{u\in\mathcal{M}(S)\,|\,&\exists \, c_l,c_u\in(0,+\infty) \\
    & \text{ s.t. }\;c_l \le u(x)\le c_u,\;\forall x\in S\big\},
\end{align*}
and define
\begin{align*}
    I^\dagger(\mu)=-\inf_{u\in\mathcal{U}'}\int_S\log\frac{Ku}{u}d\mu.
\end{align*}
The rate function $I(\cdot)$ in Theorem~\ref{TheoremMN} satisfies
\begin{align*}
    I(\mu)=I^\dagger(\mu),
\end{align*}
for all $\mu\in\mathcal{P}(S)$.
\end{proposition}

\begin{proof}
    The proof of this Proposition is analogous to the proof of Lemma 2.6 in \cite{donsker1975asymptotic}.

    Let $\mu\in\mathcal{P}(S)$. By Proposition~\ref{thm:alternativeRepresentation},
    \begin{align}
    \label{rate_func_prop}
        I(\mu) =-\inf_{u\in\mathcal{U}}\int_S\log\frac{Ku}{u}d\mu.
    \end{align}
    Since $\mathcal{U}\subset\mathcal{U}'$,
    \begin{equation}
    \label{ineq1}
        I(\mu) \le -\inf_{u\in\mathcal{U}'}\int_S\log\frac{Ku}{u}d\mu.
    \end{equation}
    Let 
    \begin{align*}
        \mathcal{U}''=\left\{u\in\mathcal{U}'\,|\,\int_S\log\frac{Ku}{u}d\mu\ge -I(\mu)\right\}.
    \end{align*}

    By \eqref{rate_func_prop}, we have that
    \begin{align*}
        \int_S\log\frac{Ku}{u}d\mu\ge -I(\mu),
    \end{align*}
    for all $u\in\mathcal{U}$,
    and since $\mathcal{U}\subset\mathcal{U}'$, it follows that $\mathcal{U}\subset\mathcal{U}''$.

    By Lusin's theorem, for every $u\in\mathcal{U}'$, there exists a sequence of continuous functions $\{u_n\}\subset \mathcal{C}(S)$ that converges to $u$ in probability with respect to $\mu$. Since for a fixed $u\in\mathcal{U}'$ there exist $c_l,c_u\in\mathbb{R}$ such that $0<c_l\le u\le c_u<+\infty$, we may assume that the same holds for $\{u_n\}$ with the same constants $c_l$ and $c_u$. That is, $\{u_n\}\subset \mathcal{U}\subset\mathcal{U}''$. Because $\lim_{n\to\infty}u_n=u$ in probability, there exists a subsequence $\{n_k\}$ such that $\lim_{k\to\infty}u_{n_k}=u$, $\mu$-almost surely. By the bounded convergence theorem, since $\{u_n\}\subset \mathcal{U}''$, then 
    \[-I(\mu)\le\lim_{k\to\infty}\int_S\log\frac{Ku_{n_k}}{u_{n_k}}d\mu=\int_S\log\frac{Ku}{u}d\mu,\]
    therefore also the limit function $u$ is an element of $\mathcal{U}''$. This implies that $\mathcal{U}'\subset\mathcal{U}''$ and we obtain, in particular, that $\mathcal{U}'=\mathcal{U}''$. As a consequence,
    \begin{align*}
        \int\log\frac{Ku}{u}d\mu\ge -I(\mu),
    \end{align*}
    for all $u\in\mathcal{U}'$, and therefore,
    \begin{align*}
         \inf_{u\in\mathcal{U}'}\int\log\frac{Ku}{u}d\mu\ge -I(\mu).
    \end{align*}
    This combined with \eqref{ineq1} implies that
    \begin{align*}
        I(\mu)=-\inf_{u\in\mathcal{U}'}\int\log\frac{Ku}{u}=I^\dagger(\mu),
    \end{align*}
    completing the proof.
\end{proof}

We can now update Theorem~\ref{TheoremMN} adding the alternative representation of the rate function from Proposition~\ref{thm:alternativeRepresentation}, together with its equivalent formulation provided in Proposition~\ref{prop:measurable_func}.

\begin{theorem}
\label{thm:update}
    Let $\{X_i\}_{i\ge0}$ be the Metropolis–Hastings chain and $K(x,dy)$ the associated transition kernel. Let $\{L^n\}_{n\ge1}\subset\mathcal{P}(S)$ be the corresponding sequence of empirical measures. Under Assumptions (A.1)–(A.3) (see \cite[Section 3]{milinanni2024a}), $\{L^n\}_{n\ge1}$ satisfies a large deviation principle with speed $n$ and rate function
    \begin{align}
        I(\mu) &=\inf_{\gamma\in A(\mu)}R(\gamma\parallel\mu\otimes K)\label{old:gamma}\\
        &=\inf_{q\in \mathcal{Q}(\mu)}\int_SR( q(x,\cdot)\parallel K(x,\cdot))\mu(dx) \label{oldRF}\\ 
        &=\sup_{u\in\mathcal{U}}-\int_S\log\frac{Ku}{u}d\mu \label{newRF}\\
        &=\sup_{u\in\mathcal{U}'}-\int_S\log\frac{Ku}{u}d\mu\label{newRFmeasurable}.
    \end{align}
\end{theorem}

In the following, we will refer to the rate function representations \eqref{old:gamma} and \eqref{oldRF} as the \textit{relative entropy representations}, whereas we will indicate \eqref{newRF} and \eqref{newRFmeasurable} as the \textit{Donsker-Varadhan representations}.

\section{Rate function upper and lower bounds}
\label{sec:bounds}

All representations of the rate function in Theorem~\ref{thm:update} require solving an optimization problem: to compute \eqref{old:gamma} and \eqref{oldRF} we need to minimize over probability measures $\gamma\in A({\mu})$ and transition kernels in $\mathcal{Q}(\mu)$, respectively, whereas for \eqref{newRF} and \eqref{newRFmeasurable} we have to maximize over functions in $\mathcal{U}$ and $\mathcal{U}'$, respectively. These tasks can be very demanding. Nevertheless, from these different representations, we can derive upper and lower bounds for the rate function, that still provide valuable information on $I(\mu)$.

\subsection{Rate function upper bounds from the relative entropy representation}
We start by showing how to obtain upper bounds for the rate function from the relative entropy representation \eqref{oldRF}.

\begin{corollary}[Rate function upper bound]
\label{cor:upper_bound}
Let $I(\mu)$ be the rate function of the large deviation principle in Theorem~\ref{thm:update}. Let let $\bar q(x,dy)$ be any probability kernel that has $\mu$ as invariant distribution, i.e., $\bar q\in\mathcal{Q}(\mu)$. Then,
    \begin{equation}
    \label{upper_bound_q}
       I(\mu)\le \int_SR(\bar q(x,\cdot)|| K(x,\cdot))\mu(dx).
    \end{equation}
\end{corollary}

\begin{proof}
    The inequality \eqref{upper_bound_q} follows directly from the relative entropy representation of the rate function~\eqref{oldRF}: because $\bar q\in \mathcal{Q}(\mu)$, we have that
    \begin{equation*}
    \begin{split}
        I(\mu) &= \inf_{ q\in\mathcal{Q}(\mu)}\int_SR( q(x,\cdot)\parallel K(x,\cdot))\mu(dx) \\
        &\le \int_SR(\bar q(x,\cdot)\parallel K(x,\cdot))\mu(dx).
    \end{split}
    \end{equation*}

\end{proof}

Based on Corollary~\ref{cor:upper_bound}, any stochastic kernel $\bar q\in\mathcal{Q}(\mu)$ determines an upper bound $\int R(\bar q(x,\cdot)\parallel K(x,\cdot))\mu(dx)$ for the rate function $I(\mu)$. In the following two Propositions, we provide explicit upper bounds by considering specific choices of kernels $\bar q\in\mathcal{Q}(\mu)$.

\begin{proposition}[Upper bound by the independent transition kernel]
\label{prop:upper_bound_indep}
    Let $\mu\ll\pi$ and let $\mu(x)$ be the density of $\mu$ with respect to the Lebesgue measure on $S$. The rate function $I(\mu)$ in the large deviation principle from Theorem~\ref{thm:update} satisfies
    \begin{align*}
        I(\mu)\le \iint \log\left(\frac{\mu(y)}{a(x,y)}\right)\mu(y)\mu(x)\,dy\,dx.
    \end{align*}
\end{proposition}

\begin{proof}
    To obtain this upper bound we apply Corollary~\ref{cor:upper_bound} with the independent transition kernel $\bar q(x,dy)=\mu(dy)$, where ``independent'' refers to the fact that it does not depend on $x$. This kernel is an element of $\mathcal{Q}(\mu)$, i.e., it has $\mu$ as invariant measure, because for every $A\in\mathcal{B}(S)$,
    \begin{align*}
        \int_S\bar q(x,A)\mu(dx)=\int_S \mu(A)\mu(dx)=\mu(A).
    \end{align*}
    
    For a fixed $x\in S$, the relative entropy between the independent kernel $\mu(dy)$ and the MH transition kernel $K(x,dy)$ is given by
    \begin{align*}
    R(\bar q(x,dy)\parallel K(x,dy)) &= R(\mu(dy)\parallel K(x,dy)) \\
    &= \int_S \log \left(\frac{\mu(y)}{a(x,y)}\right)\mu(y)dy.
\end{align*}
The upper bound \eqref{upper_bound_q} thus becomes
\begin{align*}
    \int_S &R(\bar q(x,\cdot)|| K(x,\cdot))\mu(dx)\\
    &=\iint \log\left(\frac{\mu(y)}{a(x,y)}\right)\mu(y)\mu(x)dy\,dx.
\end{align*}
    
\end{proof}

\begin{remark}
    In Proposition~\ref{prop:upper_bound_indep}, we assume that $\mu$ is absolutely continuous with respect to $\pi$, otherwise we would obtain $R(\mu(dy)\parallel K(x,dy))=+\infty$, which would lead to a trivial bound ($+\infty$).
\end{remark}

\begin{proposition}[Upperbound by the MH transition kernel]
\label{prop:upper_bound_MH}
    Let $\mu\ll\pi$ and let $\mu(x)$ be the density of $\mu$ with respect to the Lebesgue measure on $S$. Let
    \begin{align}
    \label{eq:bar_a}
        \bar a(x,y) = \min\left\{1,\frac{\mu(y)J(x|y)}{\mu(x)J(y|x)}\right\}J(y|x),
    \end{align}
    and
    \begin{align}
    \label{eq:bar_r}
        \bar r(x) = 1- \int_S\bar a (x,y)dy.
    \end{align}
    
    The rate function $I(\mu)$ in the large deviation principle from Theorem~\ref{thm:update} satisfies
    \begin{align*}
        I(\mu)\le &\iint \log\left(\frac{\bar a(x,y)}{a(x,y)}\right)\bar a(x,y)\mu(x)dy\,dx \\
        &+ \int \log\left(\frac{\bar r(x)}{r(x)}\right)\bar r(x)\mu(dx).
    \end{align*}
\end{proposition}

\begin{proof}
    For this upper bound we apply Corollary~\ref{cor:upper_bound} with the kernel $\bar q(x,dy)$ corresponding to the MH transition kernel with target $\mu(dx)$ instead of $\pi(dx)$, and same proposal probability $J(dy|x)$. The kernel that we consider is therefore
    \begin{align*}
        \bar q(x,dy) = \bar a(x,y)dy + \bar r(x)\delta_x(dy),
    \end{align*}
    with $\bar a(x,y)$ and $\bar r(x)$ defined in \eqref{eq:bar_a} and \eqref{eq:bar_r}, respectively.

    Because MH transition kernels satisfy the detailed balance condition with the corresponding target density, $\mu(dx)$ is an invariant probability distribution for $\bar q(x,dy)$, that is, $\bar q\in\mathcal{Q}(\mu)$.

    For a fixed $x\in S$, the relative entropy between $\bar q(x,dy)$ and the MH transition kernel $K(x,dy)$ is given by
    \begin{align*}
    &R(\bar q(x,dy)\parallel K(x,dy)) \\
    &= R(\bar a(x,y) dy + \bar r(x)\delta_x(dy)\parallel  a(x,y) dy +  r(x)\delta_x(dy)) \\
    &= \int_S \log \left(\frac{\bar a(x,y)}{a(x,y)}\right)\bar a(x,y)dy + \log\left(\frac{\bar r(x)}{r(x)}\right)\bar r(x).
\end{align*}
From this we obtain that the upper bound \eqref{upper_bound_q} is 
\begin{align*}
    \int R(&\bar q(x,\cdot)|| K(x,\cdot))\mu(dx) \\
    &= \iint\log\left(\frac{\bar a(x,y)}{a(x,y)}\right)\bar a(x,y)\mu(x)dy\,dx \\
    &\qquad + \int \log\left(\frac{\bar r(x)}{r(x)}\right)\bar r(x)\mu(dx).
\end{align*}
\end{proof}

\subsection{Rate function lower bound from the Donsker-Varadhan representation}

Lower bounds for the rate function from Theorem~\ref{thm:update} can be obtained directly from the Donsker-Varadhan representation of the rate function~\eqref{newRFmeasurable}, as shown in Corollary~\ref{cor:lower_bound}.

\begin{corollary}[Rate function lower bound]
\label{cor:lower_bound}
Let $I(\mu)$ be the rate function of the large deviation principle in Theorem~\ref{thm:update}. Let let $\bar u$ be any function in $\mathcal{U}'$. Then,
\begin{equation}
\label{lower_bound}
       I(\mu)\ge -\int_S \log\frac{K\bar u}{\bar u} d\mu.
    \end{equation}
\end{corollary}

\begin{proof}
Assuming $\bar u\in\mathcal{U}'$, from the Donsker-Varadhan representation of the rate function~\eqref{newRFmeasurable} we obtain   
\begin{align*}
    I(\mu)=\sup_{u\in\mathcal{U}'}-\int_S\log\frac{Ku}{u}d\mu\ge -\int_S\log\frac{K\bar u}{\bar u}d\mu.
\end{align*}
\end{proof}

Applying Corollary~\ref{cor:lower_bound} with specific choices of $\bar u$ we obtain explicit lower bounds for the rate function $I(\mu)$. One such lower bound is obtained in Proposition \ref{prop:lower_bound_derivative}.

\begin{proposition}[Lower bound by the Radon-Nikodym derivative $\frac{d\mu}{d\pi}$]
\label{prop:lower_bound_derivative}
Assume $\mu\ll\pi$ and let $\phi=\frac{d\mu}{d\pi}$ be the Radon-Nikodym derivative of $\mu$ with respect to $\pi$. Let $c_l,c_u\in\mathbb{R}$ satisfying $0<c_l<c_u<\infty$, and define

\begin{align}
    \label{phi}
    \bar\phi(x)  =(\phi(x)\vee c_l) \wedge c_u
\end{align}

Then $\bar\phi\in\mathcal{U}'$ and
\begin{align}
\label{lower_bound_phi}
    I(\mu) \ge -\int_S\log\frac{K\bar \phi}{\bar \phi}d\mu.
\end{align}
\end{proposition}

\begin{remark}
    The constants $c_l$ and $c_u$ in Proposition~\ref{prop:lower_bound_derivative} should be selected such that $c_l$ is adequately small and $c_u$ is adequately large, depending on the specific case under consideration. In a computer implementation of such lower bound, $c_l$ could coincide with the machine epsilon and $c_u$ with the floating point upper bound.
\end{remark}

\begin{proof}
    The function $\bar \phi$ given in \eqref{phi} is, by definition, measurable and bounded away from $0$ and $+\infty$, thus $\bar \phi\in\mathcal{U}'$. The lower bound \eqref{lower_bound_phi} follows directly from Corollary~\ref{cor:lower_bound}, by letting $\bar u=\bar \phi$.
\end{proof}

\begin{remark}
Note that if $\mu(x)$ denotes the Radon-Nikodym derivative of $\mu$ with respect to the Lebesgue measure on $S$, then $\phi(x)$ is given by $\phi(x)=\frac{\mu(x)}{\pi(x)}$.
\end{remark}

\subsection{Rate function lower bound by the variational formula for the relative entropy}
One further lower bound can be obtained by using the variational formula for the relative entropy, given in \cite[Proposition 2.3 (a)]{budhiraja2019analysis}, that we report here for completeness.

\begin{proposition}[Proposition 2.3 (a) in \cite{budhiraja2019analysis}]
\label{prop:variational_formula_rel_entropy}
    Suppose that $\mathcal{X}$ is a Polish space and $\theta$ a probability measure on $\mathcal{X}$. If $k$ is a measurable function mapping $\mathcal{X}$ into $\mathbb{R}$ that is bounded from below, then
    \begin{equation}
    \begin{split}
         \label{variational_formula_rel_entropy}
        -\log\int_\mathcal{X}e^{-k}&d\theta= \inf_{\gamma\in \mathcal{P}(\mathcal{X})}\left[R(\gamma\parallel \theta) + \int_\mathcal{X}k\ d\gamma\right].
    \end{split}
    \end{equation}
\end{proposition}

Using Proposition~\ref{prop:variational_formula_rel_entropy}, we obtain the following lower bound for the LDP rate function.

\begin{proposition}
    Assume that $\mu\ll\pi$ and that the Radon-Nikodym derivative $\frac{\mu(x)}{\pi(x)}$ is bounded. Then, 
    
    \begin{equation*}
        \label{lower_bound_variational}
        \begin{split}
             I(\mu)\ge -\log&\Bigg(1-\frac{1}{2}\iint \min\left\{\frac{J(y|x)}{\pi(y)},\frac{J(x|y)}{\pi(x)}\right\}\\
             &\times \left(\sqrt{\mu(x)\pi(y)}-\sqrt{\mu(y)\pi(x)}\right)^2dy\,dx\Bigg)
        \end{split}
    \end{equation*}
\end{proposition}

\begin{proof}
Let us denote the left and right hand side of \eqref{variational_formula_rel_entropy} as $\mathfrak{L}$ and $\mathfrak{R}$, respectively. Let $\mathcal{X}=S\times S$, $\theta(dx,dy)=\pi(dx)K(x,dy)=(\pi\otimes K)(dx,dy)$, and 
    \begin{align*}
        k(x,y)&=-\log\sqrt{\frac{\mu(x)}{\pi(x)}\frac{\mu(y)}{\pi(y)}}\\
        &=-\frac{1}{2}\left(\log\frac{\mu(x)}{\pi(x)}+\log\frac{\mu(y)}{\pi(y)}\right).
    \end{align*}
Because we assume that $\frac{\mu(x)}{\pi(x)}$ is bounded, i.e., there exists a constant $C>0$ such that  $\frac{\mu(x)}{\pi(x)}< C$ for all $x\in S$, then $k(x,y)>-\log C$. Therefore, the assumptions of Proposition~\ref{prop:variational_formula_rel_entropy} are satisfied.

With this choice of $\mathcal{X}$, $\theta$ and $k$, the left hand side of \eqref{variational_formula_rel_entropy} becomes
    \begin{equation*}
        \begin{split}
            \mathfrak{L} &=  -\log\iint\sqrt{\frac{\mu(x)}{\pi(x)}\frac{\mu(y)}{\pi(y)}}K(x,dy)\pi(dx)\\
        &=-\log\Bigg(\iint\sqrt{\frac{\mu(x)}{\pi(x)}\frac{\mu(y)}{\pi(y)}}a(x,dy)\pi(dx)\\
        &\qquad\quad  + \int \frac{\mu(x)}{\pi(x)}\left(1-\int a(x,dy)\right)\pi(dx)\Bigg).
        \end{split}
        \end{equation*}
In the second quality, we used the definition \eqref{MHtransitionKernel} of the MH transition kernel $K(x,dy)$, considering that the rejection probability is given by \eqref{MHrejection}. Using the definition \eqref{MHacceptance} of $a(x,y)$, we obtain
        \begin{equation*}
            \begin{split}
        \mathfrak{L} &= -\log\Bigg(1-\iint \left(\frac{\mu(x)}{\pi(x)}-\sqrt{\frac{\mu(x)}{\pi(x)}\frac{\mu(y)}{\pi(y)}}\right)\\
        &\qquad \quad \times \min\left\{1,\frac{\pi(y)J(x|y)}{\pi(x)J(y|x)}\right\}J(y|x)dy\pi(dx) \Bigg)\\
        &=-\log\Bigg(1-\iint \sqrt{\mu(x)\pi(y)}\\
        &\qquad \quad \times \left(\sqrt{\mu(x)\pi(y)}-\sqrt{\mu(y)\pi(x)}\right)\\
        &\qquad \quad \times\min\left\{\frac{J(y|x)}{\pi(y)},\frac{J(x|y)}{\pi(x)}\right\}dx\,dy\Bigg),
\end{split}
\end{equation*}
and using the symmetry of the integrand in the previous equation, we get

        \begin{equation}
            \begin{split}
            \label{leftHS}
        &\mathfrak{L}=-\log\Bigg(1-\frac{1}{2}\iint\left(\sqrt{\mu(x)\pi(y)}-\sqrt{\mu(y)\pi(x)}\right)^2\\
        &\qquad\quad\times\min\left\{\frac{J(y|x)}{\pi(y)},\frac{J(x|y)}{\pi(x)}\right\}dx\,dy\Bigg).
        \end{split}
    \end{equation}

    The right hand side of \eqref{variational_formula_rel_entropy} can be rewritten and bounded as 
    \begin{equation*}
    \begin{split}
        \mathfrak{R}&=\inf_{\gamma\in\mathcal{P}(S\times S)}\Bigg[R(\gamma\parallel \pi\otimes K) \\
        &\qquad - \frac{1}{2}\iint\left(\log\frac{\mu(x)}{\pi(x)}+\log\frac{\mu(y)}{\pi(y)}\right)\gamma(dx,dy) \Bigg]\\
        &\le \inf_{\gamma\in A(\mu)}\Bigg[R(\gamma\parallel \pi\otimes K)\\
        &\qquad - \frac{1}{2}\iint\left(\log\frac{\mu(x)}{\pi(x)}+\log\frac{\mu(y)}{\pi(y)}\right)\gamma(dx,dy) \Bigg]\\
        &=\inf_{\gamma\in A(\mu)}\Bigg[R(\gamma\parallel \pi\otimes K)- \int\log\frac{\mu(x)}{\pi(x)}\mu(dx) \Bigg].
    \end{split}
    \end{equation*}
    The inequality in the previous display holds because $A(\mu)$ is a subset of $\mathcal{P}(S\times S)$. The last equality follows from the fact that, since $\gamma(dx,dy)\in A(\mu)$, both marginals are equal to $\mu$. Note that, by the chain rule of the relative entropy (see \cite[Theorem 2.6]{budhiraja2019analysis}), if $\gamma\in A(\mu)$, then $[\gamma]_1(dx)=\mu(dx)$, and
    \begin{align*}
        R(\gamma\parallel \pi\otimes K) = &R(\mu\parallel \pi) \\
        &+ \iint R([\gamma]_{2|1}(\cdot|x)\parallel K(x,\cdot))\mu(dx).
    \end{align*}
    Therefore,
    \begin{equation}
        \label{rightHS}
        \begin{split}
             \mathfrak{R}&\le \inf_{\gamma\in A(\mu)}\iint R([\gamma]_{2|1}(\cdot|x)\parallel K(x,\cdot))\mu(dx)\\
        &=\inf_{q\in\mathcal{Q}(\mu)}\iint R(q(x,\cdot)\parallel K(x,\cdot))\mu(dx)\\
        &=I(\mu),
        \end{split}
    \end{equation}
    where we considered that, for $\gamma\in A(\mu)$, the conditional probability distributions $[\gamma]_{2|1}(\cdot|x)$ correspond to transition kernels $q\in\mathcal{Q}(\mu)$.
    
    Recalling that $\mathfrak{L}=\mathfrak{R}$, using \eqref{leftHS} and \eqref{rightHS}, we conclude that 
    \begin{equation*}
    \begin{split}
         &I(\mu)\\
         &\ge 
        -\log\left(1-\frac{1}{2}\iint \Bigg(\sqrt{\mu(x)\pi(y)}-\sqrt{\mu(y)\pi(x)}\right)^2\\
        &\qquad\times \min\left\{\frac{J(y|x)}{\pi(y)},\frac{J(x|y)}{\pi(x)}\right\}dx\,dy\Bigg).
    \end{split}
    \end{equation*}
\end{proof}

\section{Tuning Metropolis-Hastings algorithms via the rate function lower bounds}
\label{sec:tuning}
In this section, we illustrate how the bounds determined in Section~\ref{sec:bounds} can be used to tune algorithms of Metropolis-Hastings type. In particular, we will describe three schemes to perform the tuning task. In the following section we show an application of one of these schemes to the Independent Metropolis-Hastings algorithm.

Suppose that we want to sample from a target distribution $\pi(dx)$ with a proposal $J_p(y|x)dy$ depending on some hyperparameter $p\in P$ in some set $P$. In general, different choices of $p$ in the proposal $J_p(y|x)dy$ determine different Metropolis-Hastings transition kernels $K_p(x,dy)$. These, in turn, lead to Markov chains with different rate functions $I_p(\mu)$. A relevant question that arises is the following:
\emph{What hyperparameter $p$ should we choose in order to obtain the algorithm with fastest convergence?}

In light of the large deviation principle described in Section~\ref{sec:ldp} and considering Proposition~\ref{prop:min_on_sphere}, a possible way to answer this question is by implementing the tuning scheme described in Algorithm~\ref{alg:tuning_scheme_1}. Recall that $\partial B_\varepsilon(\pi)$ indicates the ball (in the Lévy-Prohorov metric) of radius $\varepsilon$, centered at $\pi$.

\begin{algorithm}
\caption{Tuning scheme for MH based on optimizing over $\partial B_\varepsilon(\pi)$}
\label{alg:tuning_scheme_1}
\begin{algorithmic}[1]
\Require Given a target distribution $\pi$.
\Require Given a family of proposal probabilities $\mathcal{J}=\{J_p(dy|x), \;p\in P\}$.
\State Fix $\varepsilon>0$.
 \State Compute
 \begin{align}
\label{optimize_p}
    p_*=\argmax_{p\in P} \min_{\mu\in \partial B_\varepsilon(\pi)}I_p(\mu).
\end{align}
\State Run the MH algorithm with proposal probability $J_{p_*}(dy|x)$.
\end{algorithmic}
\end{algorithm}

The rate function $I(\cdot)$ is given by the representations \eqref{old:gamma}, \eqref{oldRF}, \eqref{newRF} and \eqref{newRFmeasurable}, which are defined as optimization problems. As a result, \eqref{optimize_p} consists of three optimization problems, and the tuning scheme in Algorithm~\ref{alg:tuning_scheme_1} is therefore very demanding. 

In order to tune MH algorithms, instead of solving the complex problem \eqref{optimize_p}, we can utilize the lower bounds from Section \ref{sec:bounds} to obtain an \textit{indication} of what the best hyperparameter could be. Specifically, we can implement a tuning scheme as described in Algorithm~\ref{alg:tuning_scheme_2}.

\begin{algorithm}
\caption{Tuning scheme for MH based on maximizing a rate function lower bound $\ell_p(\mu)$}
\label{alg:tuning_scheme_2}
\begin{algorithmic}[1]
\Require Given a target distribution $\pi$.
\Require Given a family of proposal probabilities $\mathcal{J}=\{J_p(dy|x), \;p\in P\}$.
\State Determine a rate function lower bound $\ell_p(\cdot)$.
\State Fix a probability measure $\mu \in\mathcal{P}(S)\setminus\{\pi\}$.
 \State Compute
\begin{align}
\label{simpler_opt_p}
     p_*=\argmax_{p\in P} \ell_p(\mu).
\end{align}
\State Run the MH algorithm with proposal probability $J_{p_*}(dy|x)$.
\end{algorithmic}
\end{algorithm}

In the tuning scheme in Algorithm~\ref{alg:tuning_scheme_2}, by using the lower bound $\ell_p(\mu)$ in place of $I(\mu)$, we eliminate the optimization required in the computation of the rate function. In addition, by fixing a measure $\mu$, we remove the minimization over $\partial B_\varepsilon(\pi)$. The resulting tuning scheme (Algorithm~\ref{alg:tuning_scheme_2}) thus consists of only optimizing over $p\in P$, and therefore is easier to implement compared to the tuning scheme in Algorithm~\ref{alg:tuning_scheme_1}.

The reason why we maximize the lower bound over $p\in P$, instead of the upper bound, is that in so doing we guarantee that the algorithm associated with $p_*$ converges at a rate that is \textit{at least} the one indicated by the lower bound. 

In Section~\ref{sec:IMH} we will apply the tuning scheme in Algorithm~\ref{alg:tuning_scheme_2} to optimize hyperparameters in the IMH algorithm.

Despite being simple to implement, the tuning scheme in Algorithm~\ref{alg:tuning_scheme_2} entails a major issue: for most MH algorithms the choice of $\mu$ will affect the value of the optimizer, that is, different test measures $\mu_1\neq\mu_2$ lead, in general, to different optimal hyperparameters $p_{*,1}\neq p_{*,2}$. An improvement to this approach, that partially addresses this problem, is to maximize the lower bound taking into consideration $N>1$ test measures $\mu_i, i=1,\dots,N$, satisfying $\text{dist}_{LP}(\mu_i,\pi)\approx \varepsilon$. This is formalized in the tuning scheme in Algorithm~\ref{alg:tuning_scheme_3}, which can be considered as an intermediate scheme between the previous two.

\begin{algorithm}
\caption{Tuning scheme for MH based on maximizing a rate function lower bound $\ell_p(\cdot)$ considering multiple test measures $\mu_i$}
\label{alg:tuning_scheme_3}
\begin{algorithmic}[1]
\Require Given a target distribution $\pi$.
\Require Given a family of proposal probabilities $\mathcal{J}=\{J_p(dy|x), \;p\in P\}$.
\State Determine a rate function lower bound $\ell_p(\cdot)$.
\State Fix $\varepsilon>0$.
\State Fix $N$ test probability measures $\mu_i \in\mathcal{P}(S)\setminus\{\pi\},i=1,\dots,N$, satisfying $\text{dist}_{LP}(\mu_i,\pi)\approx 
    \varepsilon$.
 \State Compute
\begin{align}
\label{medium_simpler_opt_p}
     p_*=\argmax_{p\in P} \min_{i=1,\dots,N}\ell_p(\mu_i).
\end{align}
\State Run the MH algorithm with proposal probability $J_{p_*}(dy|x)$.
\end{algorithmic}
\end{algorithm}

In fact, we could require that all $\mu_i$ satisfy $\text{dist}_{LP}(\mu_i,\pi)\ge \varepsilon$, instead of $\text{dist}_{LP}(\mu_i,\pi)\approx \varepsilon$. However, this could be disadvantageous in computational applications. Indeed, as shown in Proposition\ref{prop:min_on_sphere}, the infimum is to be found on the boundary $\partial B_\varepsilon(\pi)$ of the ball. Consequently, if for some fixed $i$ the measure $\mu_i$ has a distance from $\pi$ much larger compared to the other test measures $\mu_j, \,j\neq i$, i.e., 
\begin{align*}
    \text{dist}_{LP}(\mu_i,\pi) \gg \text{dist}_{LP}(\mu_j,\pi),\qquad j\neq i,
\end{align*}
this could result in a much larger lower bound, that is,
\begin{align*}
    \ell_p(\mu_i) \gg \ell_p(\mu_j),\qquad j\neq i.
\end{align*}
As a result, the minimum in \eqref{medium_simpler_opt_p} will hardly be achieved at $\mu_i$. This undermines the computational effort put in the computation of $\ell_p(\mu_i)$. In conclusion, in order to avoid the waste of computational resources, we recommend implementing tuning scheme in Algorithm~\ref{alg:tuning_scheme_3} with all $\mu_i$ satisfying $\text{dist}_{LP}(\mu_i,\pi)\approx \varepsilon$.

To implement the tuning scheme in Algorithm~\ref{alg:tuning_scheme_3}, we need a (programming) function to approximate the Lévy-Prohorov distance between probability measures. An important direction for future research involves the design of an efficient algorithm to perform this.

Although the approximation of the Lévy-Prohorov metric can make this scheme very costly, the tuning scheme in Algorithm~\ref{alg:tuning_scheme_3} is to be preferred over that in Algorithm~\ref{alg:tuning_scheme_2} when tuning rather complex algorithms, e.g., MALA and HMC, because it is a closer approximation of the more precise tuning scheme in Algorithm~\ref{alg:tuning_scheme_1}.

With the three tuning schemes described here we showed how to tune hyperparameters in algorithms of MH type using the large deviation principle from \cite{milinanni2024a}, and we illustrated how this can be implemented in practice. However, the three tuning schemes in Algorithms~\ref{alg:tuning_scheme_1}-\ref{alg:tuning_scheme_3} are only some of the possible ways to use the large deviations result to optimize MH algorithms, and variations to the three schemes are possible.

\section{An illustrative example: Tuning the Independent Metropolis-Hastings algorithm}
\label{sec:IMH}
In this section we present an application of the tuning scheme in Algorithm~\ref{alg:tuning_scheme_2} to the Independent Metropolis-Hastings (IMH) algorithm. Although in Section~\ref{sec:tuning} we argued that the tuning scheme in Algorithm~\ref{alg:tuning_scheme_3} is preferable, its implementation requires an algorithm to compute the Lévy-Prokhorov distance, which we leave for future research.

The Independent Metropolis-Hastings (IMH) is an algorithm of MH type where the proposal distribution $J(y|x)dy$ does not depend on the current state $x$ of the chain. Therefore, we will drop the dependency on $x$, and denote the proposal density as $\hat J(y)$.

Let us consider the IMH algorithm with target distribution $\pi\sim\mathcal{N}(0,1)$ and proposal $\hat J \sim\mathcal{N}(m,s^2)$ with hyperparameters $m\in\mathbb{R}$ and $s\in\mathbb{R}_+$. Our goal is to determine the values of the hyperparameters $m$ and $s$ that give rise to the IMH algorithm with fastest convergence. By the tuning scheme in Algorithm~\ref{alg:tuning_scheme_2}, we determine the optimal hyperparameter $(m_*,s_*)$ by computing
\begin{align}
\label{opt_ms}
    ( m_*,  s_*) = \argmax_{(m,s)\in \mathbb{R}\times\mathbb{R}_+} \ell_{(m,s)}(\mu).
\end{align}
We apply the tuning scheme several times, considering different lower bounds $\ell_{(m,s)}(\cdot)$ of the rate function and different probability measures $\mu$. More specifically, as $\ell_{(m,s)}(\mu)$ we consider both lower bounds $\ell_{1,(m,s)}(\mu)$ from Proposition~\ref{prop:lower_bound_derivative} and $\ell_{2,(m,s)}(\mu)$ from Proposition~\ref{prop:variational_formula_rel_entropy}, and we determine the optimal hyperparameter $(m_*,s_*)$ by solving \eqref{opt_ms} numerically $6$ times, by testing the following $6$ different probability measures:
\begin{itemize}[itemsep=1ex, leftmargin=*, label=\tiny$\bullet$]
    \item $\mu_1\sim\mathcal{N}(1,2^2) $,
    \item $\mu_2\sim \text{Weib}(3,2)$, %shape and scale
    \item $\mu_3\sim\text{Unif}([0,1])$,
    \item $\mu_4\sim \frac{1}{2}\mathcal{N}(5,2^2)+\frac{1}{2}\mathcal{N}(-3,1^2)$,
    \item $\mu_5\sim \text{Exp}(1)$,
    \item $\mu_6\sim \text{Gamma}(3,2)$. %shape and rate
\end{itemize}
For all measures $\mu_i,i=1,\dots,6$, we evaluated the lower bounds $\ell_{1,(m,s)}(\mu_i)$ and $\ell_{2,(m,s)}(\mu_i)$ for values $(m,s)$ in the rectangle $[-3,3]\times[0.2,3]$ (more specifically, on a grid with increments $0.1$ in both directions). To carry out these computations, we used the R programming language. Interestingly, for all $i=1,\dots,6$, the maximum value of either lower bounds was achieved in correspondence of $(m,s)=(0,1)$. Figures~\ref{fig:LB1} and~\ref{fig:LB2} (generated in R) show contour plots of the lower bounds $\ell_{1,(m,s)}(\mu_i)$ and $\ell_{2,(m,s)}(\mu_i)$, respectively, for $i=1,\dots,6$, and $(m,s)\in [-3,3]\times[0.2,3]$. The maxima are indicated with dashed white lines.

\begin{figure*}
    \centering
    \begin{minipage}{0.32\textwidth}
        \centering
        \includegraphics[width=\textwidth]{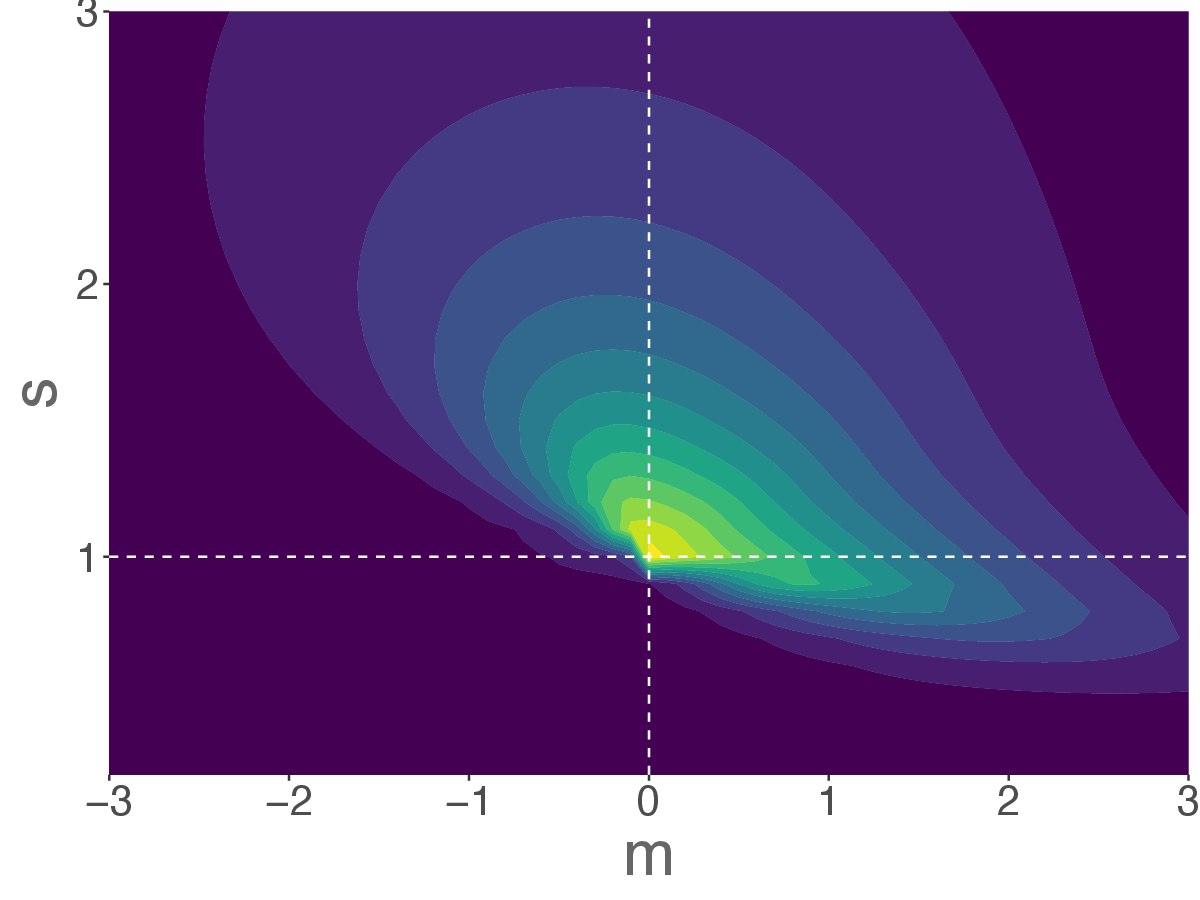}
        \subcaption{$\mu_1\sim\mathcal{N}(1,2^2) $}
    \end{minipage}
    \begin{minipage}{0.32\textwidth}
        \centering
        \includegraphics[width=\textwidth]{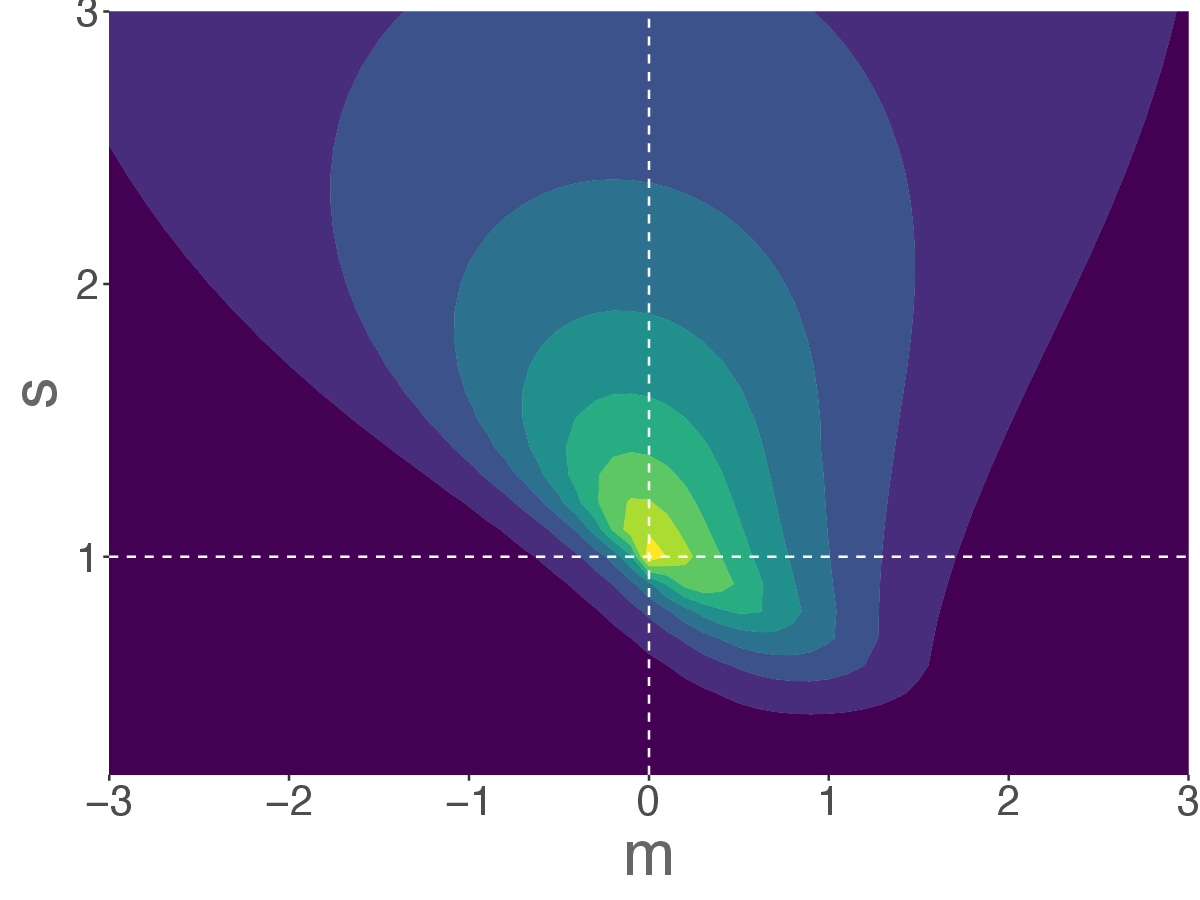}
        \subcaption{$\mu_2\sim \text{Weib}(3,2)$}
    \end{minipage}
    \begin{minipage}{0.32\textwidth}
        \centering
        \includegraphics[width=\textwidth]{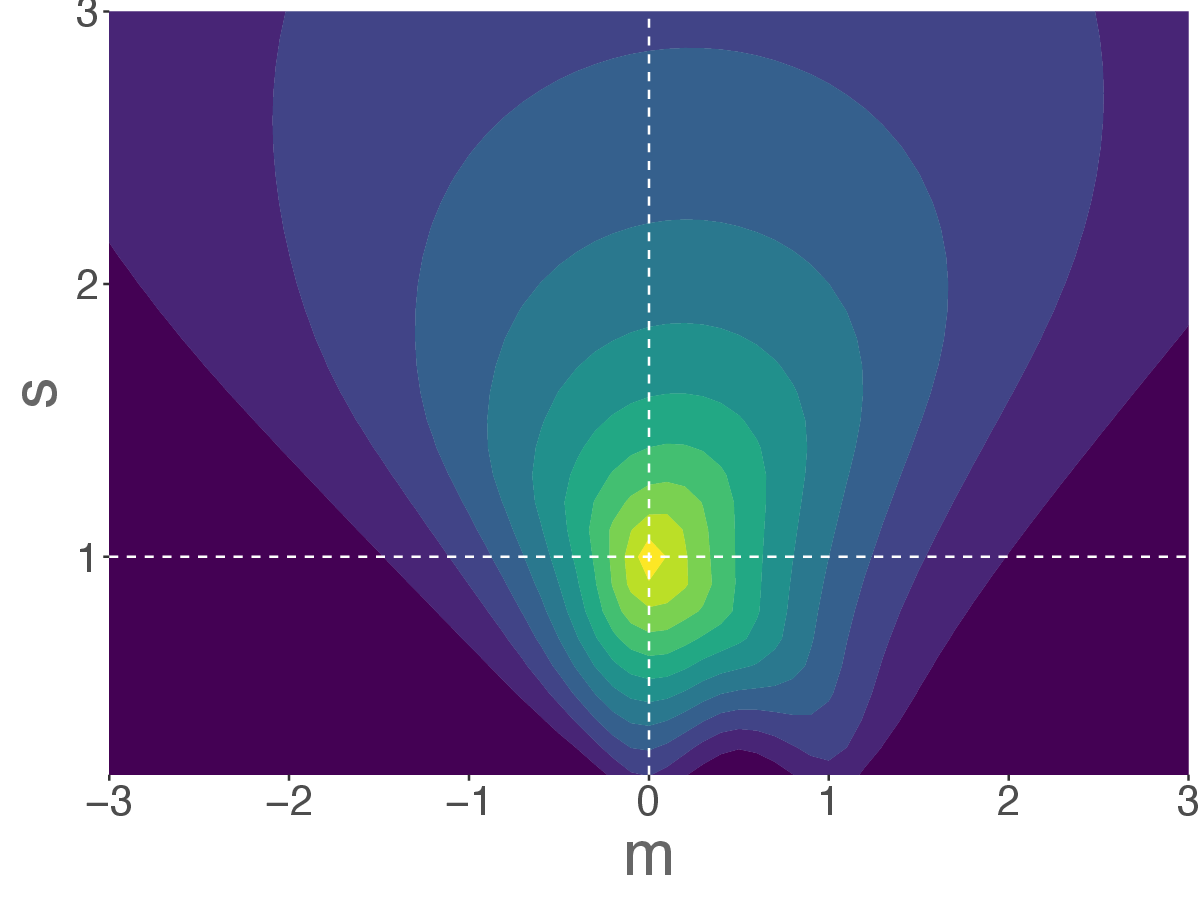}
        \subcaption{$\mu_3\sim\text{Unif}([0,1])$}
    \end{minipage}

    \vspace{0.5cm}

    \begin{minipage}{0.32\textwidth}
        \centering
        \includegraphics[width=\textwidth]{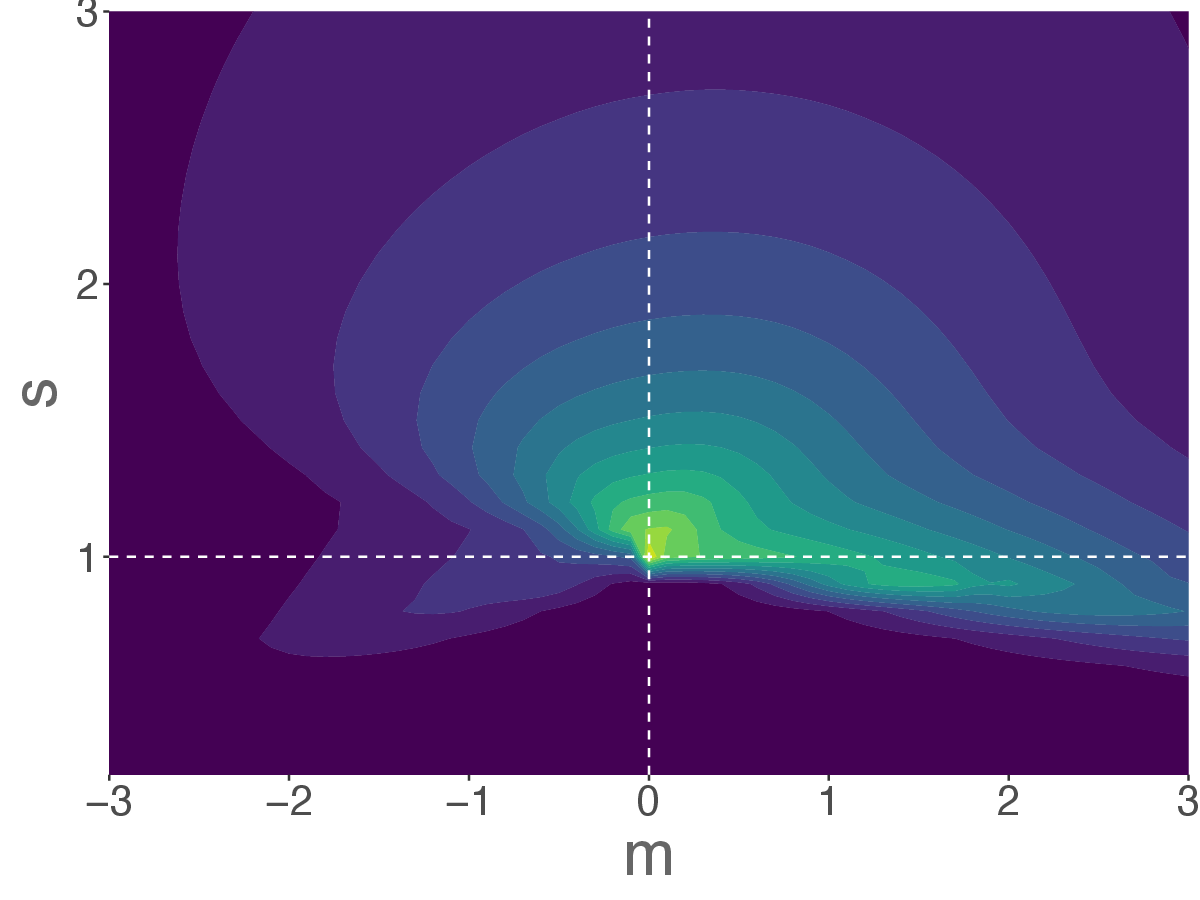}
        \subcaption{\hspace{-1mm}$\mu_4\hspace{-.5mm}\sim\hspace{-.5mm}\frac{1}{2}(\mathcal{N}(5,2^2)+\mathcal{N}(-3,1))$}
    \end{minipage}
    \begin{minipage}{0.32\textwidth}
        \centering
        \includegraphics[width=\textwidth]{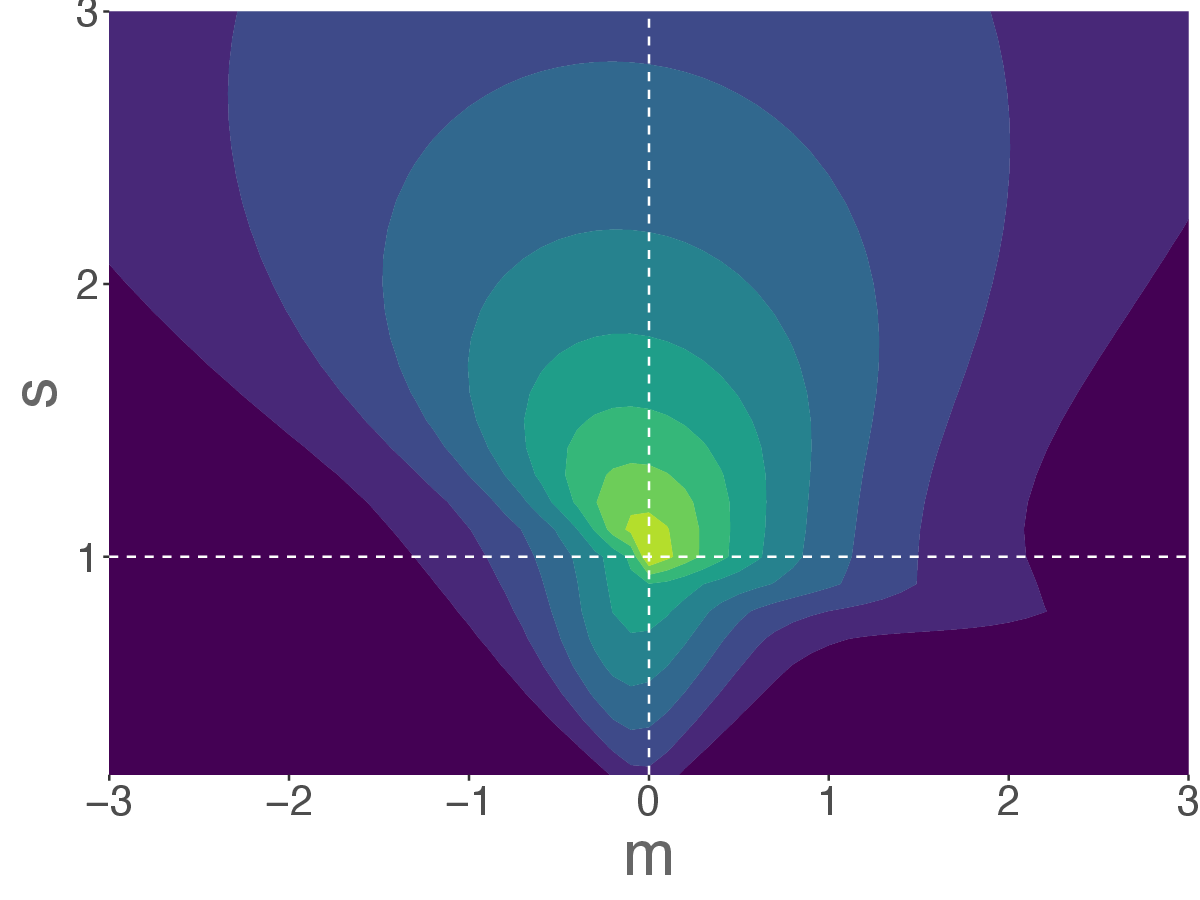}
        \subcaption{$\mu_5\sim \text{Exp}(1)$}
    \end{minipage}
    \begin{minipage}{0.32\textwidth}
        \centering
        \includegraphics[width=\textwidth]{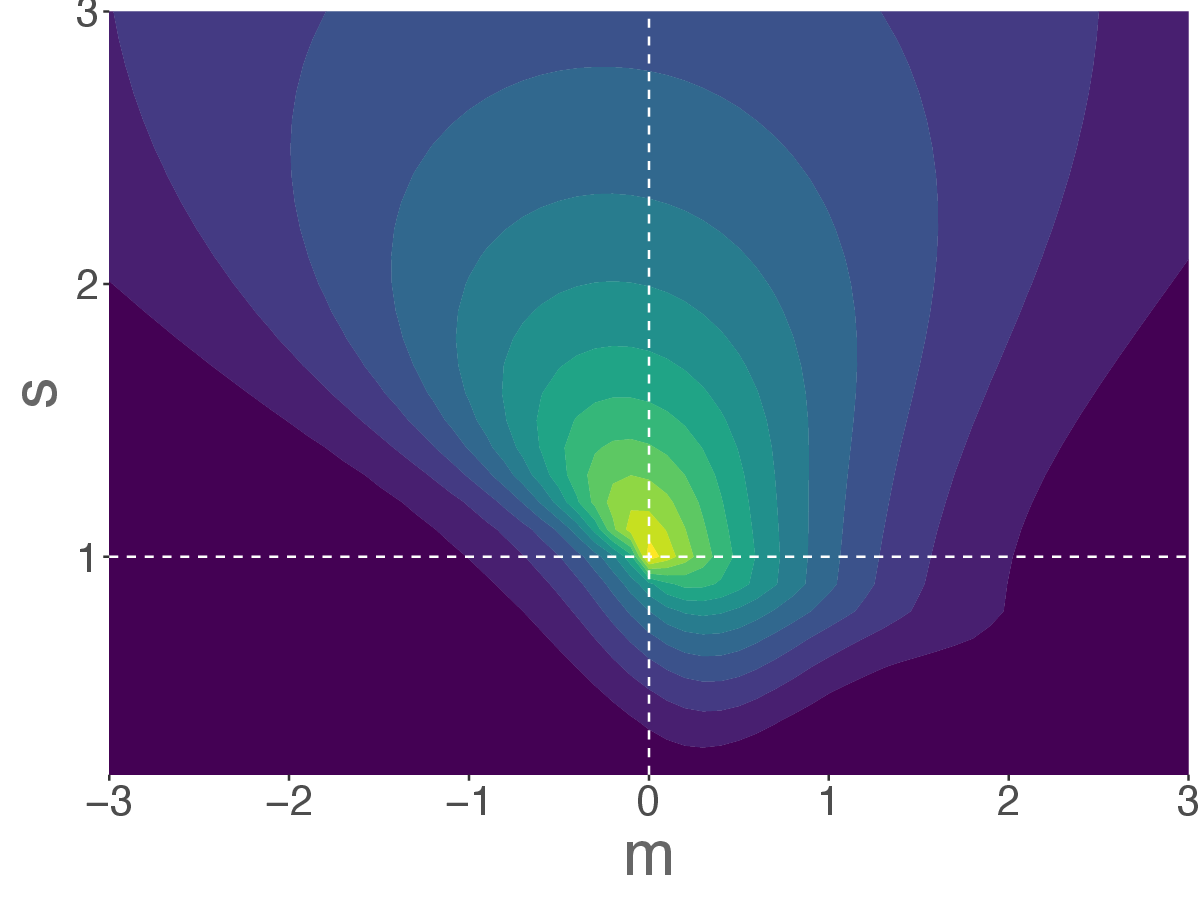}
        \subcaption{$\mu_6\sim \text{Gamma}(3,2)$}
    \end{minipage}
    
    \caption{Contour plot of the lower bound from Proposition~\ref{prop:lower_bound_derivative} as a function of $(m,s)$, for different choices of $\mu$. The $x$ and $y$ axes represent values of $m$ and $s$, respectively, and white dashed lines indicate the maximum of the lower bound, which is achieved at $(m,s)=(0,1)$ for all $\mu_i$, $i=1,\dots,6$. The plots were generated using the R programming language.}
\label{fig:LB1}
\end{figure*}

\begin{figure*}
    \centering
    \begin{minipage}{0.32\textwidth}
        \centering
        \includegraphics[width=\textwidth]{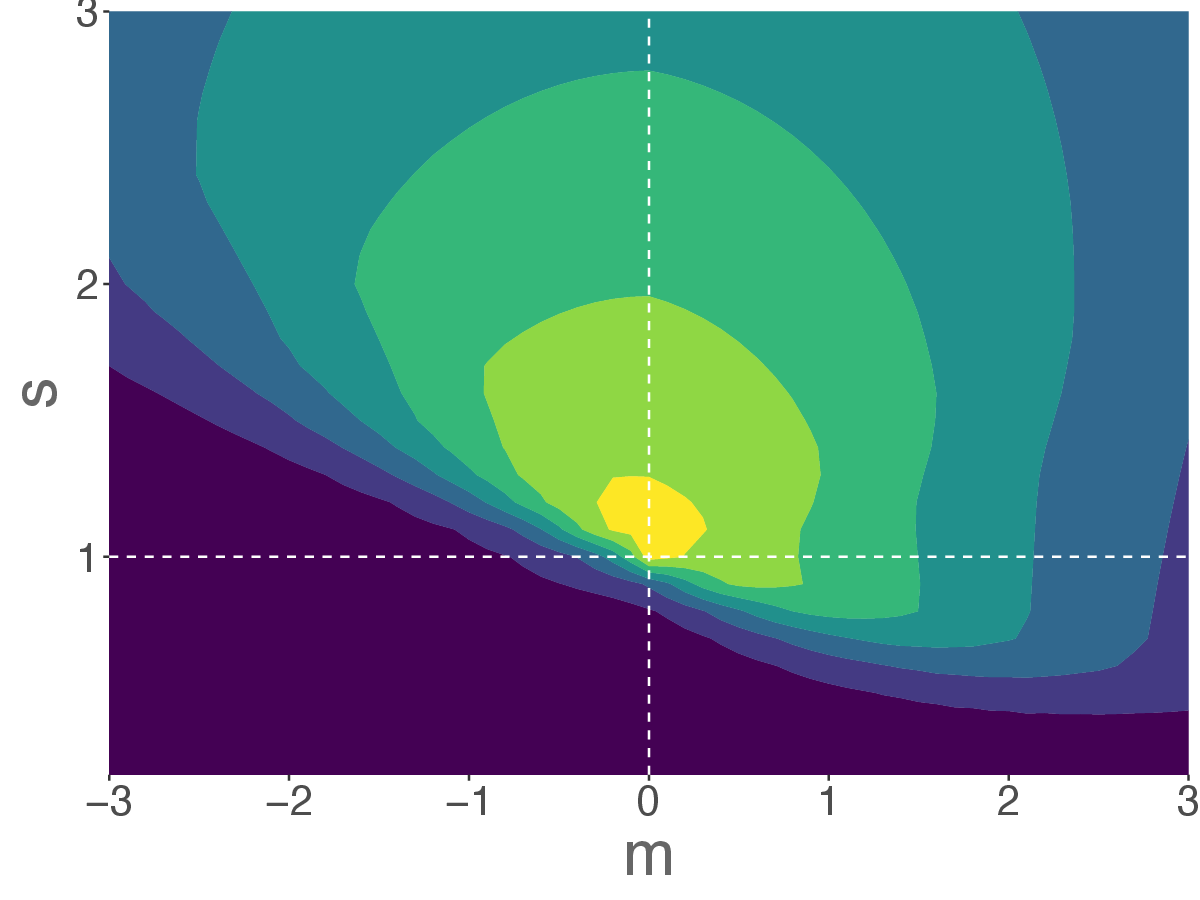}
        \subcaption{$\mu_1\sim\mathcal{N}(1,2^2) $}
    \end{minipage}
    \begin{minipage}{0.32\textwidth}
        \centering
        \includegraphics[width=\textwidth]{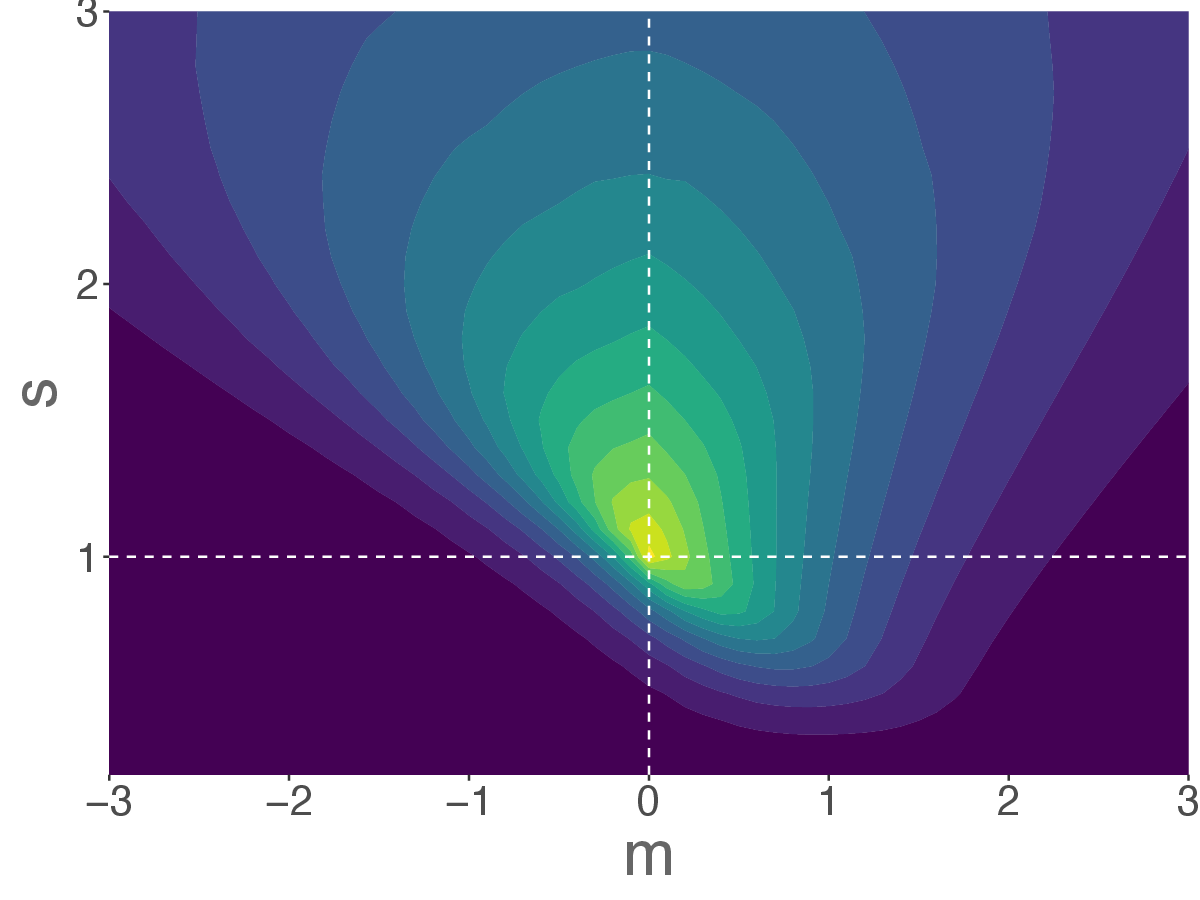}
        \subcaption{$\mu_2\sim \text{Weib}(3,2)$}
    \end{minipage}
    \begin{minipage}{0.32\textwidth}
        \centering
        \includegraphics[width=\textwidth]{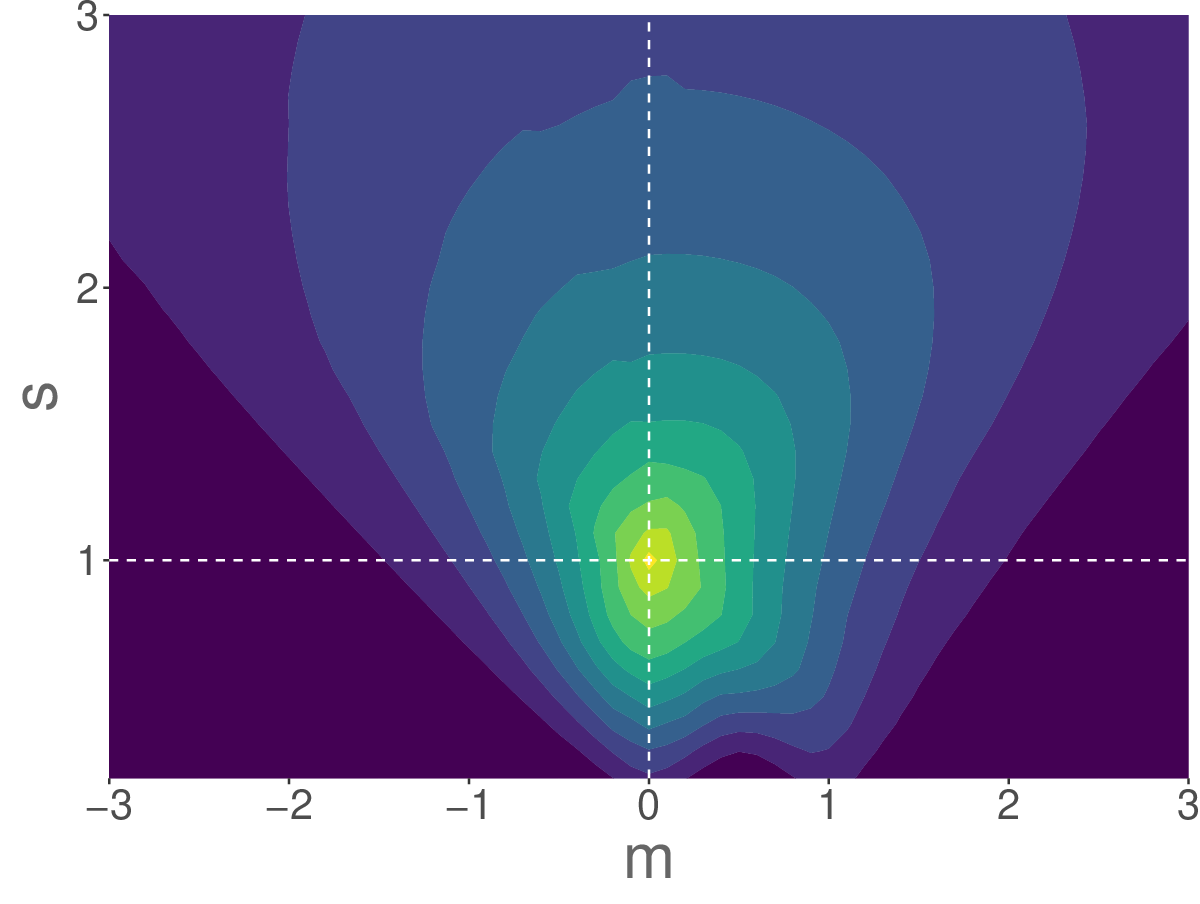}
        \subcaption{$\mu_3\sim\text{Unif}([0,1])$}
    \end{minipage}

    \vspace{0.5cm}

    \begin{minipage}{0.32\textwidth}
        \centering
        \includegraphics[width=\textwidth]{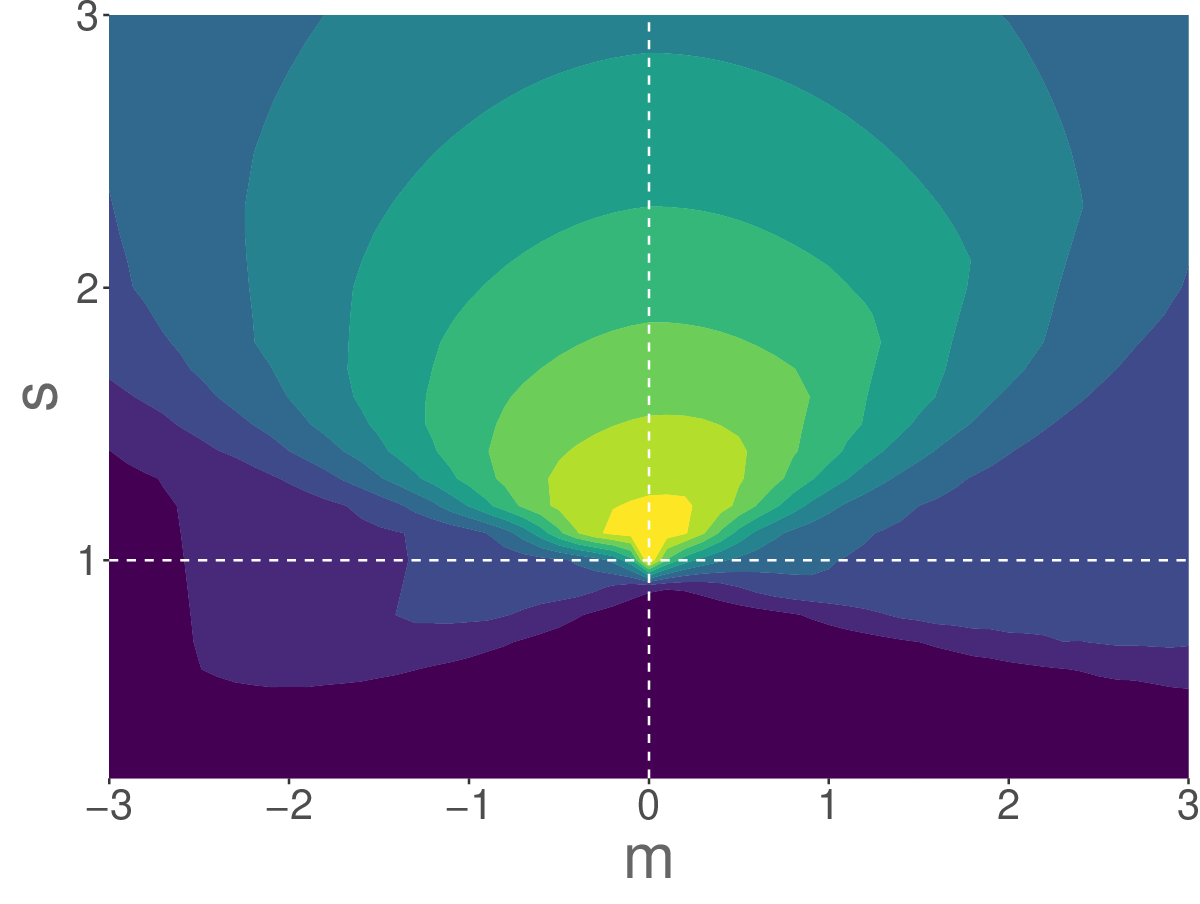}
        \subcaption{\hspace{-1mm}$\mu_4\hspace{-.5mm}\sim\hspace{-.5mm}\frac{1}{2}(\mathcal{N}(5,2^2)+\mathcal{N}(-3,1))$}
    \end{minipage}
    \begin{minipage}{0.32\textwidth}
        \centering
        \includegraphics[width=\textwidth]{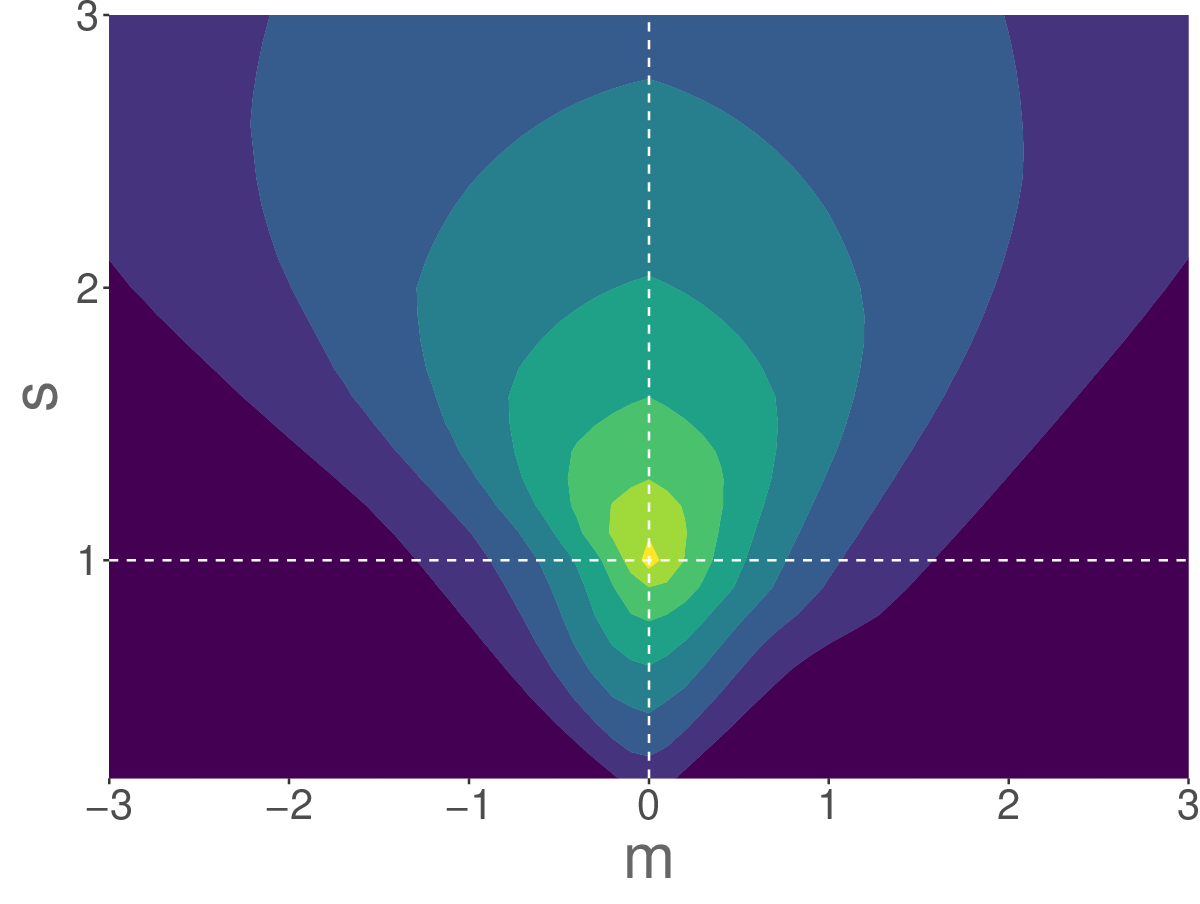}
        \subcaption{$\mu_5\sim \text{Exp}(1)$}
    \end{minipage}
    \begin{minipage}{0.32\textwidth}
        \centering
        \includegraphics[width=\textwidth]{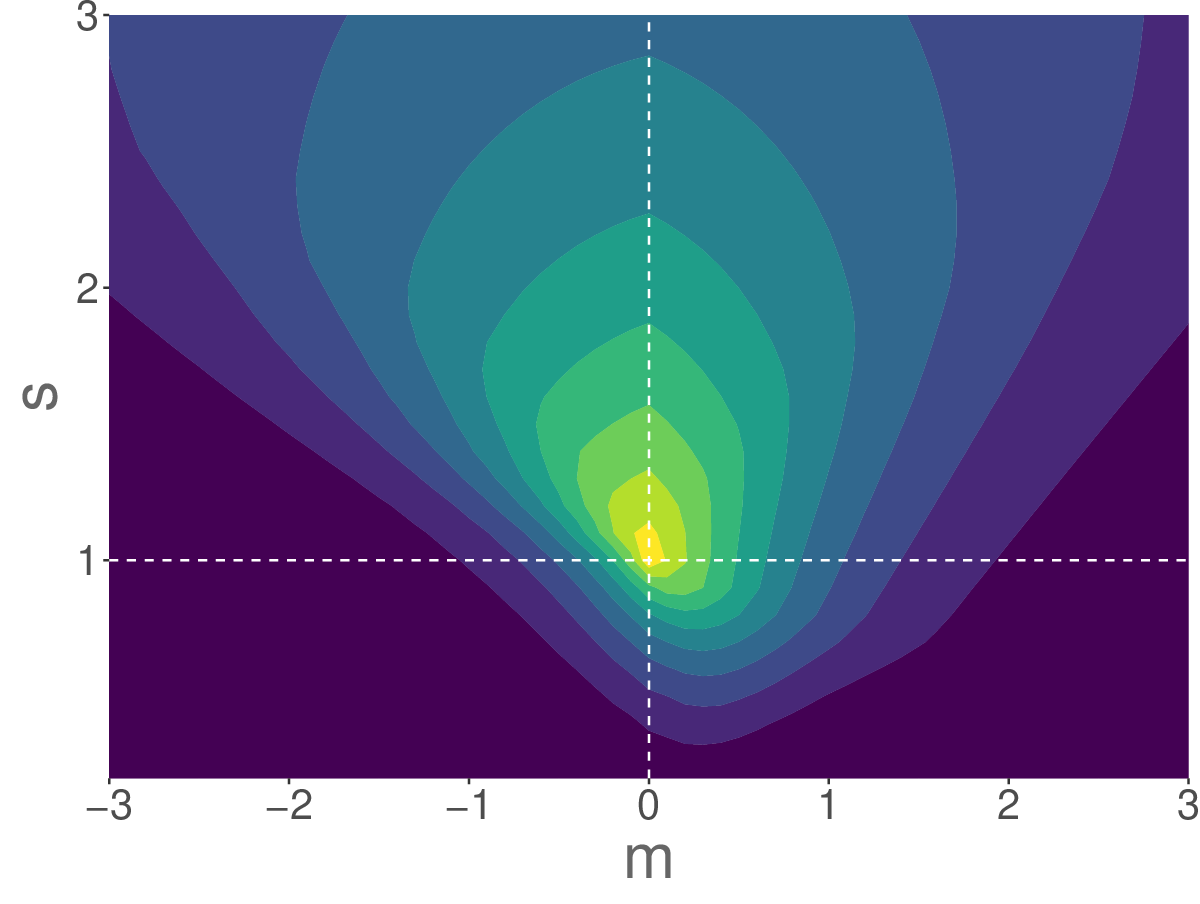}
        \subcaption{$\mu_6\sim \text{Gamma}(3,2)$}
    \end{minipage}
    
    \caption{Contour plot of the lower bound from Proposition~\ref{prop:variational_formula_rel_entropy} as a function of $(m,s)$, for different choices of $\mu$. The $x$ and $y$ axes represent values of $m$ and $s$, respectively, and white dashed lines indicate the maximum of the lower bound, which is achieved at $(m,s)=(0,1)$ for all $\mu_i$, $i=1,\dots,6$. The plots were generated using the R programming language.}
\label{fig:LB2}
\end{figure*}

This numerical results indicate that, among the independent proposal distributions $\hat J_{(m,s)} \sim\mathcal{N}(m,s^2)$, the proposal that determines the IMH algorithm with fastest convergence to the target $\pi\sim\mathcal{N}(0,1)$ is $\hat J_{(0,1)}\sim\mathcal{N}(0,1)$. This result is not surprising. In fact, when $(m,s)=(0,1)$, the proposal has the same distribution as the target $\pi$. Consequently, the MH acceptance probability is $\min\left\{1,\frac{\pi(y)J(x|y)}{\pi(x)J(y|x)}\right\}=1$, i.e., all proposals are accepted. This means that when $\hat J=\pi$, the IMH algorithm is equivalent to sampling directly from the target.

We also investigated the upper bounds obtained in Propositions~\ref{prop:upper_bound_indep} and~\ref{prop:upper_bound_MH}. However, we did not identify any interesting pattern. Figure~\ref{fig:UB} (generated in R) shows contour plots of the upper bound from Proposition~\ref{prop:upper_bound_MH}, for $i=1,\dots,6$, and $(m,s)\in [-3,3]\times[0.2,3]$. Nevertheless, upper bounds carry useful information. In fact, by means of both a lower and an upper bound we identify a closed interval containing the rate function, and therefore can provide quantitative estimates for the rate of convergence of the algorithm.
\begin{figure*}
    \centering
    \begin{minipage}{0.32\textwidth}
        \centering
        \includegraphics[width=\textwidth]{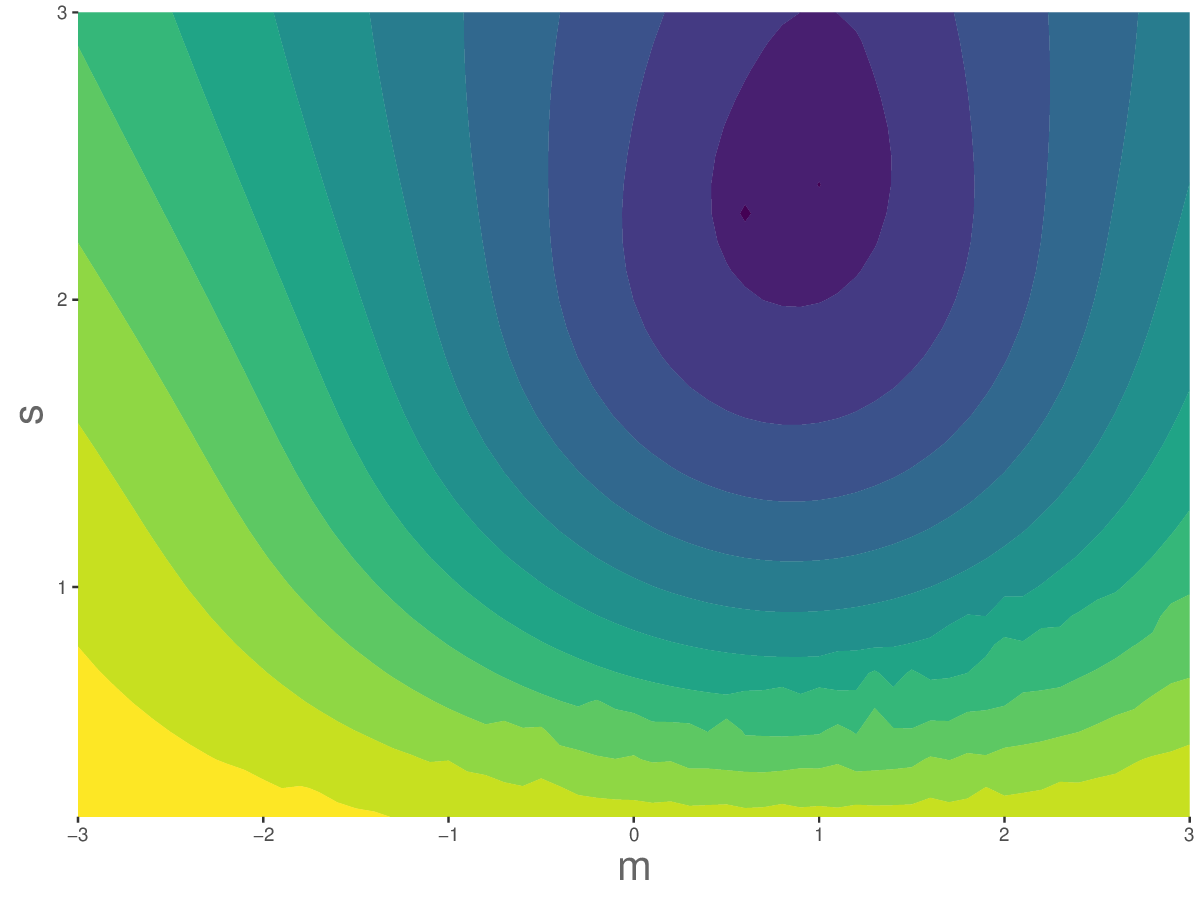}
        \subcaption{$\mu_1\sim\mathcal{N}(1,2^2) $}
    \end{minipage}
    \begin{minipage}{0.32\textwidth}
        \centering
        \includegraphics[width=\textwidth]{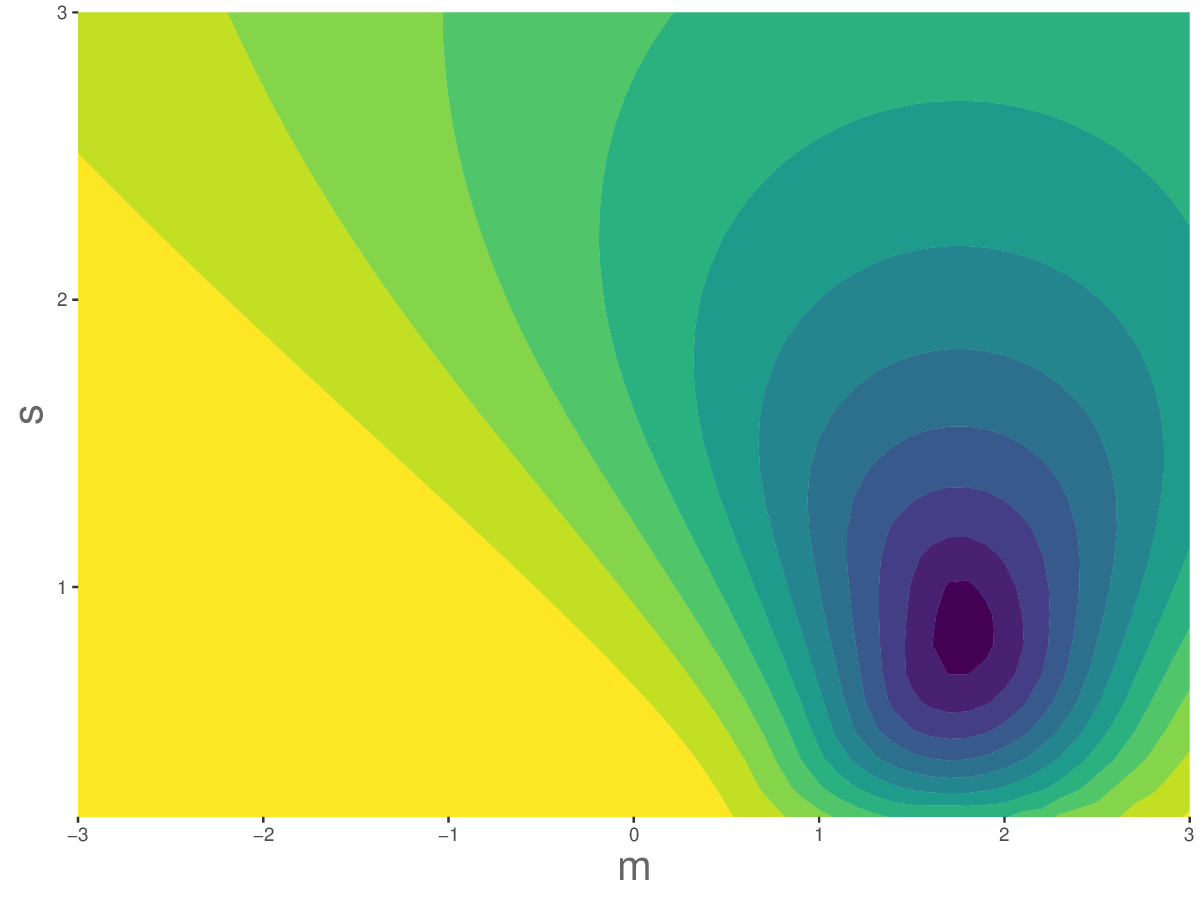}
        \subcaption{$\mu_2\sim \text{Weib}(3,2)$}
    \end{minipage}
    \begin{minipage}{0.32\textwidth}
        \centering
        \includegraphics[width=\textwidth]{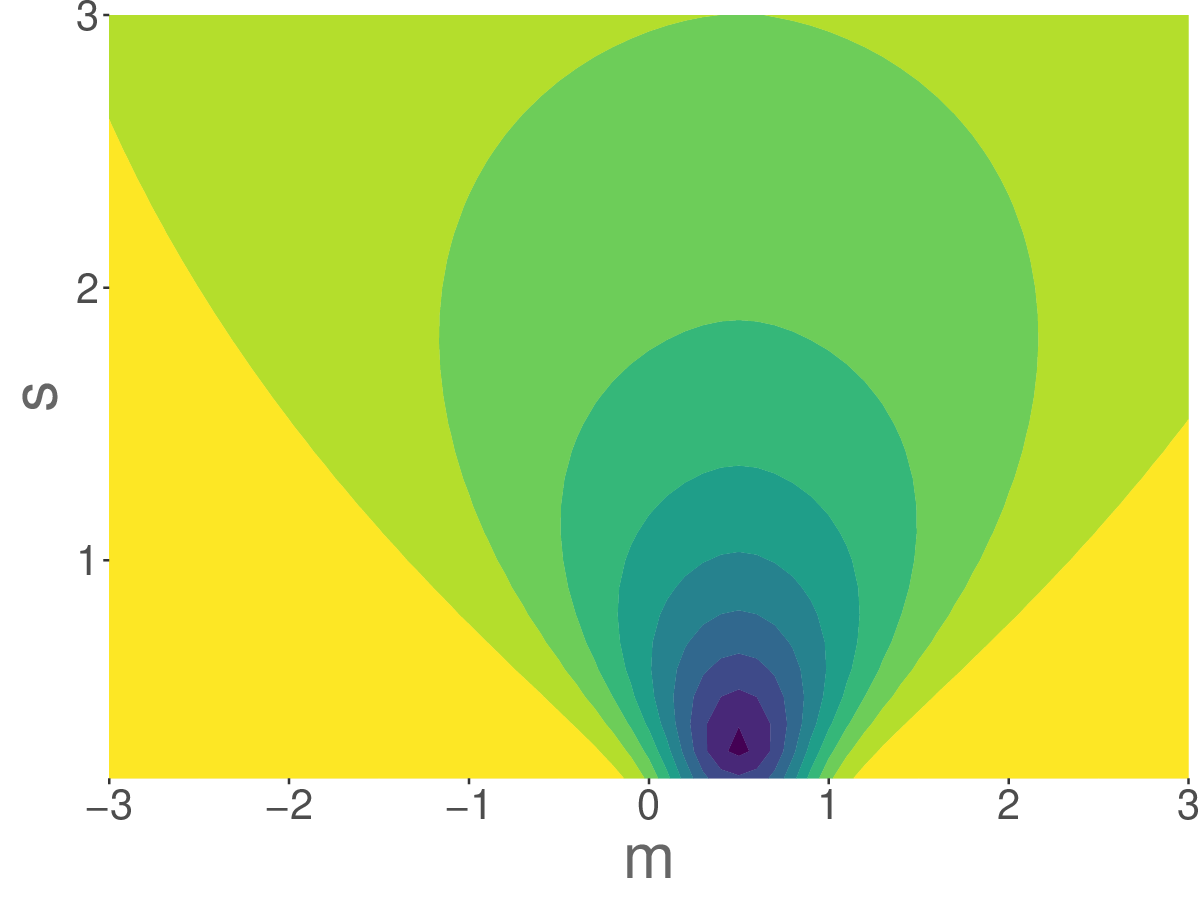}
        \subcaption{$\mu_3\sim\text{Unif}([0,1])$}
    \end{minipage}

    \vspace{0.5cm}

    \begin{minipage}{0.32\textwidth}
        \centering
        \includegraphics[width=\textwidth]{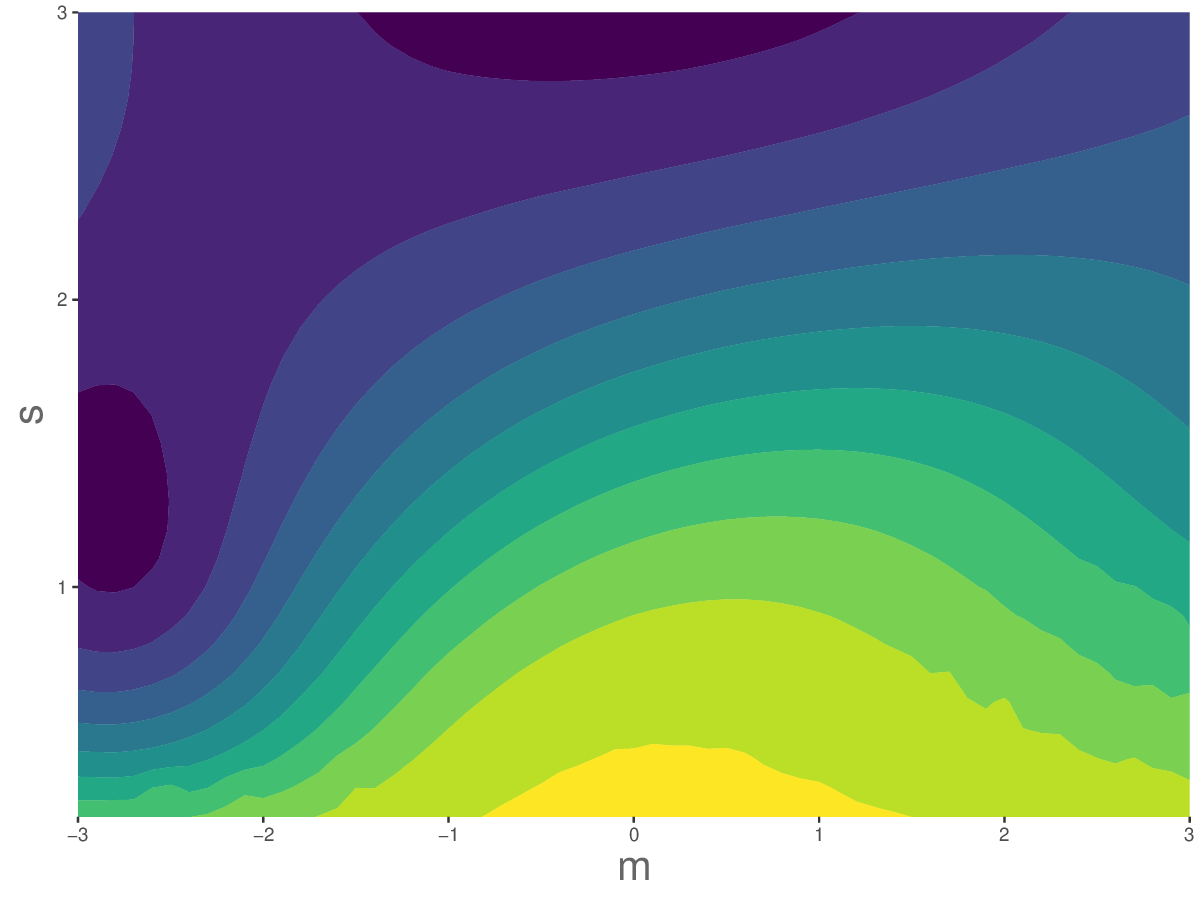}
        \subcaption{\hspace{-1mm}$\mu_4\hspace{-.5mm}\sim\hspace{-.5mm}\frac{1}{2}(\mathcal{N}(5,2^2)+\mathcal{N}(-3,1))$}
    \end{minipage}
    \begin{minipage}{0.32\textwidth}
        \centering
        \includegraphics[width=\textwidth]{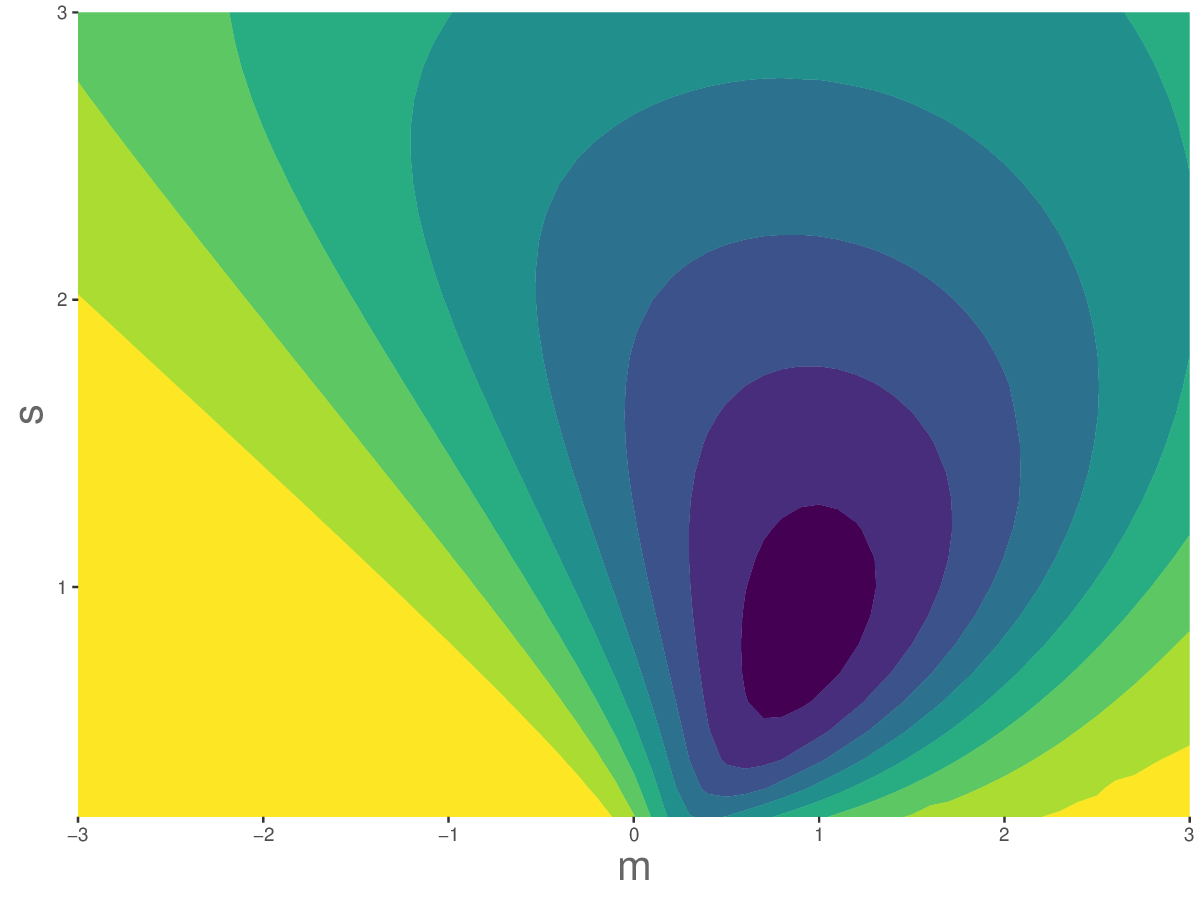}
        \subcaption{$\mu_5\sim \text{Exp}(1)$}
    \end{minipage}
    \begin{minipage}{0.32\textwidth}
        \centering
        \includegraphics[width=\textwidth]{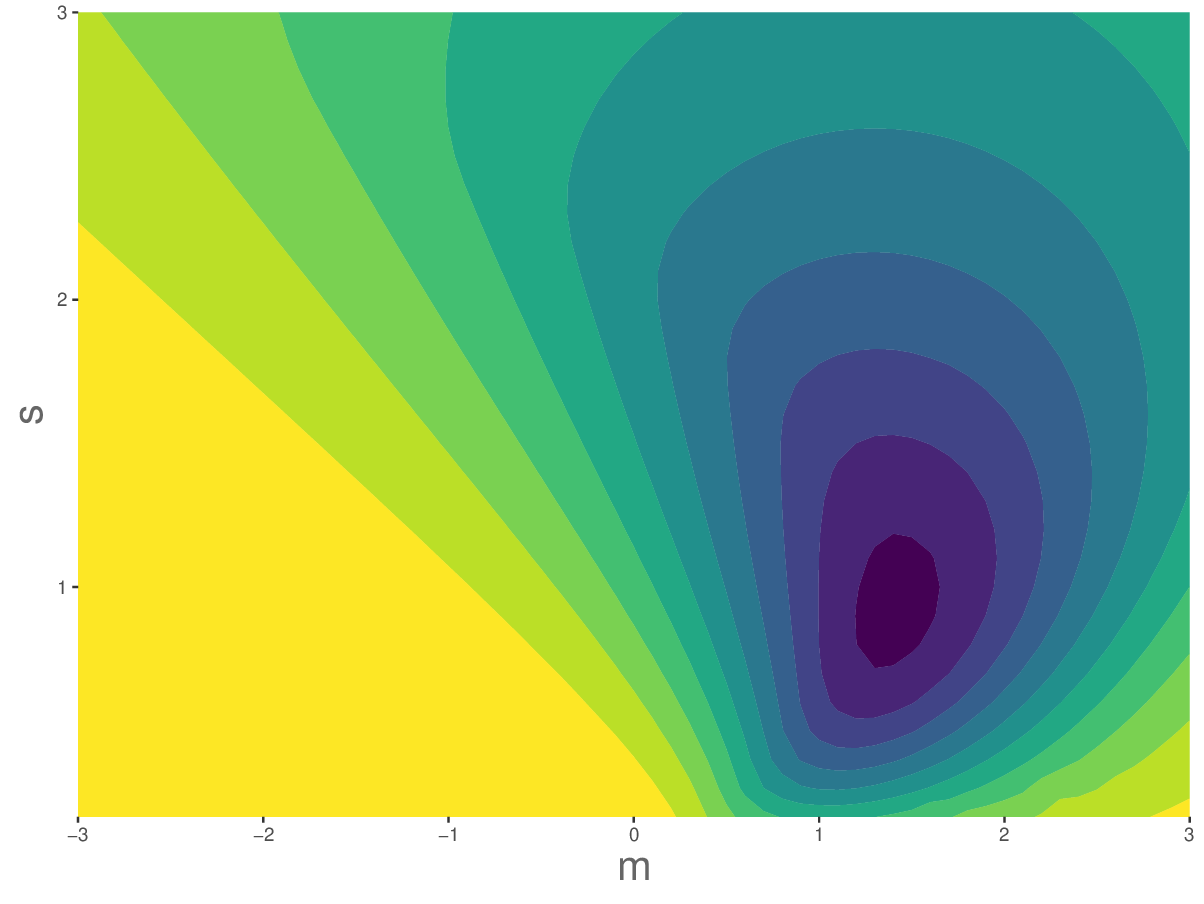}
        \subcaption{$\mu_6\sim \text{Gamma}(3,2)$}
    \end{minipage}
    
    \caption{Contour plot of the upper bound from Proposition~\ref{prop:upper_bound_MH} as a function of $(m,s)$, for different choices of $\mu$. The $x$ and $y$ axes represent values of $m$ and $s$, respectively. The plots were generated using the R programming language.}
\label{fig:UB}
\end{figure*}

\backmatter

\bmhead{Acknowledgements}

I would like to express my sincere gratitude to my PhD advisor, Prof. Pierre Nyquist (Chalmers \& Gothenburg University), for his mentorship and insightful feedback during the development of this work.

This research was supported by the Swedish e-Science Research Centre (SeRC).

\bibliography{bibliography}

\end{document}